\documentclass[11pt,a4paper,reqno]{amsart}

\usepackage{amsmath,amsthm,amssymb}
\usepackage{bm}

\numberwithin{equation}{section}

\theoremstyle{plain}
\newtheorem{thm}{Theorem}[section]
\newtheorem{lem}[thm]{Lemma}
\newtheorem{prop}[thm]{Proposition}
\newtheorem{cor}[thm]{Corollary}

\theoremstyle{definition}
\newtheorem{defi}[thm]{Definition}

\theoremstyle{remark}
\newtheorem{rem}[thm]{Remark}

\begin{document}

\title[Variational problems for integral invariants]
{Variational problems for integral invariants of the second fundamental form of a map between pseudo-Riemannian manifolds}

\author[R. Akiyama \and T. Sakai \and Y. Sato]{Rika Akiyama \and Takashi Sakai \and Yuichiro Sato}

\address{Department of Mathematical Sciences,
Tokyo Metropolitan University,
Minami-Osawa 1-1, Hachioji, Tokyo, 192-0397, Japan}

\email{akiyama-rika@ed.tmu.ac.jp}

\address{Department of Mathematical Sciences,
Tokyo Metropolitan University,
Minami-Osawa 1-1, Hachioji, Tokyo, 192-0397, Japan}

\email{sakai-t@tmu.ac.jp}

\address{Academic Support Center, Kogakuin University, 
Nakano-cho, 2665-1, Hachioji, Tokyo, 192-0015, Japan}

\email{kt13699@ns.kogakuin.ac.jp}



\thanks{
This work was partly supported by Osaka City University Advanced Mathematical Institute :
MEXT Joint Usage/Research Center on Mathematics and Theoretical Physics JPMXP0619217849.
The first author was partly supported by JST SPRING, Grant Number JPMJSP2156.
The second author was partly supported by JSPS KAKENHI Grant Numbers 21K03250.
The third author was partly supported by Foundation of Research Fellows, The Mathematical Society of Japan.} 

\begin{abstract}
We study variational problems for integral invariants,
which are defined as integrations of invariant functions of the second fundamental form,
of a smooth map between pseudo-Riemannian manifolds.
We derive the first variational formulae for integral invariants
defined from invariant homogeneous polynomials of degree two.
Among these integral invariants,
we show that the Euler--Lagrange equation of the Chern--Federer energy functional
is reduced to a second order PDE.
Then we give some examples of Chern--Federer submanifolds in Riemannian space forms.
\end{abstract}

\maketitle

\section{Introduction}
\label{sec:introduction}

The theory of harmonic maps and biharmonic maps is one of the important fields in differential geometry.
Recall that a smooth map $\varphi : (M,g_M) \rightarrow (N,g_N)$ between Riemannian manifolds
is said to be \textit{harmonic} if it is a critical point of the energy functional
\[
E(\varphi ) = \frac{1}{2}\int _M |d\varphi |^2 d\mu _{g_M}.
\]
By the first variational formula,
then $\varphi$ is a harmonic map if and only if
\begin{equation} \label{Euler-Lagrange equation of harmonic map}
\tau (\varphi ) = \mathrm{tr}_{g_M} (\widetilde{\nabla } d\varphi) = 0,
\end{equation}
where $\widetilde{\nabla} d\varphi$ is the second fundamental form
and $\tau (\varphi)$ is the tension field  of $\varphi$.
The Euler--Lagrange equation (\ref{Euler-Lagrange equation of harmonic map}) is a second order nonlinear PDE,
therefore the theory of harmonic maps has been developed in geometric analysis,
furthermore it is investigated applying methods of integrable systems.
As a generalization of harmonic maps,
Eells and Lemaire \cite{EL} introduced the notion of \textit{biharmonic map},
which is a critical point of the bienergy functional
\[
E_2(\varphi )= \frac{1}{2} \int _M |\tau (\varphi )| ^2 d\mu _{g_M}.
\]
Jiang \cite{Jiang} showed that $\varphi$ is a biharmonic map if and only if
\[
\tau _2(\varphi ) = -\overline{\nabla }^*\overline{\nabla }\tau (\varphi ) - \mathrm{tr}_{g_M} R^N\left( d\varphi (\cdot), \tau (\varphi ) \right) d\varphi (\cdot) = 0,
\]
where $-\overline{\nabla }^*\overline{\nabla } $ is the rough Laplacian
and $R^N$ is the Riemannian curvature tensor of $(N,g_N)$.
By definition, it is clear that a harmonic map is biharmonic.
One of the important problems in the study of biharmonic maps is Chen's conjecture,
that is, an arbitrary biharmonic submanifold of a Euclidean space must be minimal.

On the other hand, in integral geometry,
Howard \cite{Howard} provided integral invariants of submanifolds
by using invariant polynomials of the second fundamental form,
and then he formulated the kinematic formula in Riemannian homogeneous spaces (see also \cite{KSS}).
In his formulation, there are some notable integral invariants of submanifolds.
One is integral invariants in the Chern--Federer kinematic formula.
These integral invariants played significant roles in differential geometry.
For example,
Weyl \cite{Weyl} showed that the volume of a tube around a compact submanifold in a Euclidean space
can be represented as a polynomial of the radius of the tube,
where the coefficients are integral invariants of the second fundamental form of the submanifold.
Also, Allendoerfer and Weil \cite{AW} used these integral invariants
to describe the extended Gauss--Bonnet theorem,
and this leads to the development of the theory of characteristic classes.
Another notable one is the integral invariant defined from
a certain invariant homogeneous polynomial of degree two.
This invariant polynomial also appears in the definition of the Willmore-Chen invariant,
which is a conformal invariant of submanifolds (\cite{Chen1, Chen2}).

In Section~\ref{sec:integral invariants},
with an idea of integral geometry,
we introduce integral invariants of a smooth map $\varphi : (M,g_M) \rightarrow (N,g_N)$ between pseudo-Riemannian manifolds
by using invariant functions of the second fundamental form of $\varphi$.
In particular, we focus on integral invariants of $\varphi$ defined from invariant homogeneous polynomials of degree two.
The space of those polynomials is spanned by
the square norm of the second fundamental form and the square norm of the tension field,
which are denoted by $\mathcal{Q}_1$ and $\mathcal{Q}_2$ respectively.
Hence, here the family of integral invariants includes the bienergy functional.
In this paper, we study variational problems for these integral invariants of $\varphi$.
In Section~\ref{sec:first variational formulae},
we derive the first variational formulae for $\mathcal{Q}_1$- and $\mathcal{Q}_2$-energy functionals.
By the linearity, then we have the first variational formulae for all integral invariants of degree two.
Note that it implies an alternative expression of the Euler--Lagrange equation of the bienergy functional.
As mentioned above, from the viewpoint of integral geometry,
there are two notable polynomials, called the Chern--Federer polynomial and the Willmore--Chen polynomial,
in the space of invariant homogeneous polynomials of degree two.
In Section~\ref{sec:CF maps},
we discuss some properties of the Chern--Federer energy functional from the viewpoint of variational problems.
The Euler--Lagrange equation of an integral invariant of degree two is a fourth order PDE in general,
however, we show that the Euler--Lagrange equation of the Chern--Federer energy functional is reduced to a second order PDE.
In Section~\ref{subsec:alternative expression 2},
we describe a symmetry of the Euler--Lagrange equation of the Chern--Federer energy functional
comparing with a symmetry of the Chern--Federer polynomial.
In Section~\ref{sec:CF submanifolds},
we give some examples of Chern--Federer submanifolds in Riemannian space forms.
Here, a Chern--Federer submanifold is the image of an isometric immersion which is a Chern--Federer map.
For an isometric immersion into a Riemannian space form,
a necessary and sufficient condition to be a Chern--Federer map is described in Theorem~\ref{CF_isom_imm}.
Considering this condition,
there is an obstruction for the domain manifold.
In addition, as a trivial example,
we can see that any isometric immersion of a Ricci-flat manifold into a Euclidean space is a Chern--Federer map.
Finally, we discuss isometric immersions of flat tori into the $3$-sphere
and isoparametric hypersurfaces in Riemannian space forms.

\section{Integral invariants of a map between pseudo-Riemannian manifolds}
\label{sec:integral invariants}

In this section, we define integral invariants of the second fundamental form of a map between pseudo-Riemannian (or semi-Riemannian) manifolds. 
An $m$-dimensional pseudo-Euclidean space with index $p$ is denoted by $\mathbb{E}^m_p = (\mathbb{R}^m,\left\langle , \right\rangle )$ with $\left\langle x, y\right\rangle = -\sum _{i=1}^p x_i y_i + \sum_{j=p+1}^m x_jy_j\; (x,y\in \mathbb{R}^m)$. 
Define $\mathrm{II}(\mathbb{E}^m_p,\mathbb{E}^n_q)$ to be
\[
\mathrm{II}(\mathbb{E}^m_p,\mathbb{E}^n_q) := \left\{ H : \mathbb{E}^m_p \times \mathbb{E}^m_p \rightarrow \mathbb{E}^n_q \; ;\; \text{symmetric bilinear map} \right\},
\]
which is a $\frac{1}{2}nm(m+1)$-dimensional vector space. 
Let $G$ be the direct product group of pseudo-orthogonal groups defined by
\[
G:= O(p,m-p)\times O(q,n-q). 
\]
The group $G$ acts on $\mathrm{II}(\mathbb{E}^m_p,\mathbb{E}^n_q)$, 
that is for $g=(a,b)\in G$ and $H\in \mathrm{II}(\mathbb{E}^m_p,\mathbb{E}^n_q)$ then $gH$ is given by
\[
(gH) (u,v) := b\left( H(a^{-1}u, a^{-1}v) \right) \quad \quad (u,v\in \mathbb{E}^m_p). 
\]
Then a function $\mathcal{P}$ on $\mathrm{II}(\mathbb{E}^m_p,\mathbb{E}^n_q)$ is said to be \textit{$G$-invariant}
if $\mathcal{P}(gH) = \mathcal{P}(H)$ for all $g\in G$ and $H \in \mathrm{II}(\mathbb{E}^m_p,\mathbb{E}^n_q)$.

Let $(M^m_p,g_M)$ and $(N^n_q,g_N)$ be pseudo-Riemannian manifolds,
and $\varphi : M \rightarrow N$ a $C^{\infty}$-map.
Thoughout this paper, a fiber metric on a vector bundle is also denoted by $\langle , \rangle$.
The second fundamental form of the map $\varphi $ is the symmetric bilinear map
$\widetilde{\nabla }d\varphi : \Gamma(TM)\times \Gamma(TM) \rightarrow \Gamma(\varphi ^{-1}TN)$ defined by
\[
(\widetilde{\nabla }d\varphi )(X,Y) :=\overline{\nabla }_X\left( d\varphi (Y) \right) - d\varphi \left( \nabla _XY \right)
\]
for any vector fields $X,Y\in \Gamma (TM)$, 
which is a section of $\bigodot ^2T^*M\otimes \varphi ^{-1}TN$. 
Here $\nabla$ is the Levi--Civita connection on the tangent bundle $TM$ of $(M_p^m, g_M)$. 
$\overline{\nabla }$ and $\widetilde{\nabla }$ are the induced connections 
on the bundles $\varphi ^{-1}TN$ and $T^*M\otimes \varphi^{-1}TN$. 
If $\varphi $ is an isometric immersion, then we have
\[
(\widetilde{\nabla }d\varphi )(X,Y) = \nabla '_{d\varphi (X)}d\varphi (Y) - d\varphi \left( \nabla _XY \right) = \nabla '_XY- \nabla _XY,
\]
where $\nabla '$ is the Levi--Civita connection on the tangent bundle $TN$ of $(N_q^n, g_N)$, 
\textrm{i.e.} the second fundamental form of the isometric immersion $\varphi $ agrees with the second fundamental form of the submanifold.  

For each $x\in M$, we can write
\[
( \widetilde{\nabla }d\varphi )_x : T_xM\times T_xM\rightarrow T_{\varphi (x)}N,
\]
which is a symmetric bilinear map. 
Let $\{e_i\} _{i=1}^m$ be a pseudo-orthonormal basis of $T_xM$, 
$\{e^i\} _{i=1}^m$ the dual basis of $\{ e_i \}$, 
and $\{ \xi _\alpha \} _{\alpha =1}^n$ a pseudo-orthonormal basis of $T_{\varphi (x)}N$. 
Hence we identify $T_xM$ and $T_{\varphi (x)}N$ with $\mathbb{E}^m_p$ and $\mathbb{E}^n_q$, respectively. 
Then $(\widetilde{\nabla}d\varphi)_x$ can be expressed as
\[
( \widetilde{\nabla }d\varphi )_x = \sum_\alpha \varepsilon '_{\alpha } \sum _{i,j} h_{ij}^\alpha \; e^i\odot e^j\otimes \xi _\alpha ,
\]
where $h_{ij}^\alpha $ is defined by
\[
h_{ij}^\alpha = \big\langle ( \widetilde{\nabla }d\varphi ) _x(e_i,e_j), \xi _\alpha \big\rangle,
\]
and
\[
\varepsilon'_\alpha = 
\begin{cases}
-1 & (\alpha =1,\cdots ,q) \\
 1 & (\alpha = q+1,\cdots ,n).
\end{cases}
\]
Thus we have a linear isomorphism between $T_x^*M\odot T_x^*M\otimes T_{\varphi (x)}N$ and $\mathrm{II}(\mathbb{E}^m_p,\mathbb{E}^n_q)$. 
That is, $( \widetilde{\nabla }d\varphi ) _x \in T_x^*M\odot T_x^*M\otimes T_{\varphi (x)}N $ corresponds to $H_x:=( h_{ij}^\alpha ) \in \mathrm{II}(\mathbb{E}^m_p,\mathbb{E}^n_q)$. 
Therefore, for a $G$-invariant function $\mathcal{P}$ on $\mathrm{II}(\mathbb{E}^m_p,\mathbb{E}^n_q)$, we define an \textit{invariant function} of the second fundamental form of $\varphi $ as follows:
\[
\mathcal{P}\big( ( \widetilde{\nabla }d\varphi )_x \big) := \mathcal{P}(H_x). 
\]
This definition 
does not depend on the choices of $\{e_i\} _{i=1}^m$ and $\{ \xi _\alpha \} _{\alpha =1}^n$ 
since $\mathcal{P}$ is $G$-invariant and a change of a basis is the action of the pseudo-orthogonal group. 
Also, $\mathcal{P}\big( ( \widetilde{\nabla }d\varphi ) _x \big)$ is a smooth function on $M$. 
\begin{defi}
Let $(M^m_p,g_M)$ be an $m$-dimensional compact pseudo-Riemannian manifold with index $p$, 
$(N^n_q,g_N)$ an $n$-dimensional pseudo-Riemannian manifold with index $q$, 
and $\mathcal{P}$ a $G$-invariant function on $\mathrm{II}(\mathbb{E}^m_p,\mathbb{E}^n_q)$. 
Then for a smooth map $\varphi : M \rightarrow N$, we define
\[
I^{\mathcal{P}}(\varphi ) := \int _M \mathcal{P}\big( ( \widetilde{\nabla }d\varphi )_x \big) d\mu _{g_M}. 
\]
We call $I^{\mathcal{P}}(\varphi)$ the \textit{integral invariant} of $\varphi$ with respect to $\mathcal{P}$.
\end{defi}
By definition,
$I^{\mathcal{P}}(\varphi )$ is an invariant of a map $\varphi$ between pseudo-Riemannian manifolds,
that is, $I^{\mathcal{P}}(g\circ \varphi \circ f^{-1}) = I^{\mathcal{P}}(\varphi)$ holds 
for any $f\in \mathrm{Isom}(M)$ and $g\in \mathrm{Isom}(N)$. 

We consider the following $G$-invariant polynomials on $\mathrm{II}(\mathbb{E}^m_p,\mathbb{E}^n_q)$. 
For $H=( h_{ij}^\alpha ) \in \mathrm{II}(\mathbb{E}^m_p,\mathbb{E}^n_q)$, define
\begin{align*}
\mathcal{Q}_1(H) := \sum _\alpha \varepsilon ' _\alpha \sum_{i,j} \varepsilon _i \varepsilon _j (h_{ij}^\alpha ) ^2 \quad \text{and} \quad \mathcal{Q}_2(H) := \sum _\alpha \varepsilon '_\alpha \left( \sum _i \varepsilon _i h_{ii}^\alpha \right) ^2
\end{align*}
with
\[
\varepsilon_i =
\begin{cases}
-1 & (i=1,\cdots, p) \\
 1 & (i=p+1,\cdots, m).
\end{cases}
\]
$\mathcal{Q}_1(H)$ and $\mathcal{Q}_2(H)$ are $G$-invariant homogeneous polynomials of degree two on $\mathrm{II}(\mathbb{E}^m_p,\mathbb{E}^n_q)$. 

\begin{defi}
For $\varphi \in C^{\infty }(M,N)$, 
the \textit{$\mathcal{Q}_1$-energy functional} $I^{\mathcal{Q}_1}(\varphi )$ 
and the \textit{$\mathcal{Q}_2$-energy functional} $I^{\mathcal{Q}_2}(\varphi )$ are defined by
\begin{align}
I^{\mathcal{Q}_1}(\varphi ) = \int _M \mathcal{Q}_1\big( (\widetilde{\nabla }d\varphi )_x \big) d\mu _{g_M} = \int _M \left\langle \widetilde{\nabla }d\varphi , \widetilde{\nabla }d\varphi  \right\rangle d\mu _{g_M} \label{Q_1}
\end{align}
and
\begin{align}
I^{\mathcal{Q}_2}(\varphi ) = \int _M \mathcal{Q}_2 \big( (\widetilde{\nabla }d\varphi )_x\big) d\mu_{g_M} = \int _M \left\langle \mathrm{tr}_{g_M} (\widetilde{\nabla }d\varphi ) , \mathrm{tr}_{g_M} (\widetilde{\nabla }d\varphi ) \right\rangle d\mu _{g_M} \label{Q_2}.
\end{align}
Then $\varphi $ is called a \textit{$\mathcal{Q}_1$-map} if it is a critical point of $I^{\mathcal{Q}_1}(\varphi )$. 
Also, then $\varphi $ is called a \textit{$\mathcal{Q}_2$-map} if it is a critical point of $I^{\mathcal{Q}_2}(\varphi )$.
\end{defi}

\begin{rem}
The $\mathcal{Q}_2$-energy functional $I^{\mathcal{Q}_2}(\varphi)$ is equal to two times of the bienergy functional $E_2(\varphi )$. 
Indeed, when $\varphi $ is a smooth map between Riemannian manifolds, 
it holds that
\[
I^{\mathcal{Q}_2}(\varphi ) = \int _M \left\langle \mathrm{tr}_{g_M}(\widetilde{\nabla }d\varphi ), \mathrm{tr}_{g_M}(\widetilde{\nabla }d\varphi ) \right\rangle d\mu _{g_M} = \int _M \left| \mathrm{tr}_{g_M} (\widetilde{\nabla }d\varphi ) \right|^2 d\mu _{g_M} =  2E_2(\varphi ).
\]
\end{rem}

\begin{rem}
When $\dim M=4$, 
the $\mathcal{Q}_1$-energy functional and $\mathcal{Q}_2$-energy functional are invariant under homothetic changes of the metric on the domain $M$. 
\end{rem}

\section{The first variational formulae of $\mathcal{Q}_1$-energy and $\mathcal{Q}_2$-energy}
\label{sec:first variational formulae}

\subsection{Preliminaries}
Let $(M^m_p,g_M)$ be an $m$-dimensional compact pseudo-Riemannian manifold with index $p$, 
$(N^n_q, g_N)$ an $n$-dimensional pseudo-Riemannian manifold with index $q$, 
and $\varphi : M \rightarrow N$ a $C^{\infty}$-map. 
In this section, we use the following notation. 

A local pseudo-orthonormal frame field of $(M^m_p,g_M)$ is a set of $m$-local vector fields $\{ e_i \} _{i=1}^m$ such that 
$g_M(e_i,e_j) = \varepsilon _i\delta _{ij}$
with $\varepsilon _1=\cdots =\varepsilon _p=-1$, 
$\varepsilon _{p+1}=\cdots =\varepsilon _m=1$.

$\widetilde{\nabla } ^2d\varphi $ and $\widetilde{\nabla }^3d\varphi$ are defined by
\[
(\widetilde{\nabla }^2d\varphi )(X,Y,Z) := \overline{\nabla }_X\big( (\widetilde{\nabla }d\varphi ) (Y,Z) \big) - ( \widetilde{\nabla }d\varphi )\left( \nabla_XY,Z \right) - (\widetilde{\nabla }d\varphi )\left( Y,\nabla _XZ \right) 
\]
and
\begin{align*}
(\widetilde{\nabla }^3d\varphi )(X,Y,Z,W) &:= \overline{\nabla }_X\big( (\widetilde{\nabla } ^2d\varphi )(Y,Z,W) \big) - (\widetilde{\nabla }^2d\varphi )(\nabla _XY,Z,W) \\
&\quad  - (\widetilde{\nabla }^2d\varphi )(Y, \nabla _XZ,W) - (\widetilde{\nabla }^2d\varphi )\left( Y,Z,\nabla _XW \right)
\end{align*}
for any vector fields $X,Y,Z,W\in \Gamma (TM)$. 
$\widetilde{\nabla }^2d\varphi$ and $\widetilde{\nabla }^3d\varphi $ are sections of $\bigotimes^3T^*M\otimes \varphi ^{-1}TN$ and $\bigotimes ^4T^*M\otimes \varphi ^{-1}TN$, respectively. 
By definition, $\widetilde{\nabla }^2d\varphi $ has the following symmetry
\[
(\widetilde{\nabla }^2d\varphi )(X,Y,Z) = (\widetilde{\nabla }^2d\varphi )(X,Z,Y).
\]

The tension field $\tau (\varphi )$ of $\varphi $ is defined by
\[
\tau (\varphi ) = \mathrm{tr}_{g_M} (\widetilde{\nabla }d\varphi ) = \sum _i \varepsilon _i(\widetilde{\nabla }d\varphi ) (e_i,e_i) = \sum _i \varepsilon _i \big( \widetilde{\nabla }_{e_i} d\varphi \big) (e_i) .
\]
If $\varphi $ is an isometric immersion, then its tension field is equal to $m$ times of the mean curvature vector field.

In general,
the curvature tensor field $R^E$ of a connection $\nabla^E$ on the bundle $E$ over $M$ is defined by
\[
R^E(X,Y):= \nabla^E_X \nabla^E_Y - \nabla^E_Y \nabla^E_X - \nabla^E_{[X,Y]} \quad (X,Y \in \Gamma (TM)).
\]
In particular, for the curvature tensor field $\widetilde{R}$ of the induced connection $\widetilde{\nabla}$
on the bundle $T^*M\otimes \varphi ^{-1}TN$, we have
\begin{align*}
\big( \widetilde{R}(X,Y) d\varphi \big) (Z) &= R^{\varphi ^{-1}TN} (X,Y)d\varphi (Z) - d\varphi \left( R^M(X,Y)Z \right) \\
&= R^N\left( d\varphi (X),d\varphi (Y) \right) d\varphi (Z) - d\varphi \left( R^M(X,Y)Z \right) \quad (X,Y,Z\in \Gamma (TM)),
\end{align*}
where $R^M$, $R^N$ and $R^{\varphi ^{-1}TN}$ are the curvature tensor fields on $TM$, $TN$ and $\varphi ^{-1}TN$, respectively. 

Then we derive the first variational formulae of the $\mathcal{Q}_1$-energy and $\mathcal{Q}_2$-energy separately.

\subsection{The first variational formula of $\mathcal{Q}_1$-energy}
\label{subsec:Q_1}

We consider a smooth variation $\{ \varphi _t \} _{t\in I}\; (I:=(-\varepsilon , \varepsilon))$ of $\varphi $,
that is we consider a smooth map $\Phi $ given by
\[
\Phi :M\times I \rightarrow N, \quad (x,t) \mapsto \Phi (x,t) =: \varphi _t(x)
\]
such that $\varphi _0(x)=\varphi (x)$ for all $x\in M$,
and denote by $V$ its variational vector field, that is 
\[
V= d\Phi \left( \left. \frac{\partial }{\partial t}\right|_{t=0} \right) \in \Gamma (\varphi ^{-1}TN) .
\]
We denote by $\bm{\nabla}$, $\bm{\overline{\nabla}}$ and $\bm{\widetilde{\nabla}}$ the induced connections on $T(M\times I)$, $\Phi^{-1}TN$ and $T^*(M\times I) \otimes \Phi ^{-1}TN$, respectively. 
Let $\{ e_i\} _{i=1}^m$ be a local pseudo-orthonormal frame field on a neighborhood $U$ of $x\in M$, 
then $\left\{ e_i, \frac{\partial }{\partial t} \right\}$ is a pseudo-orthonormal frame field on the neighborhood $U\times I$
of $(x,t) \in M \times I$,
and it holds that
\[
\bm{\nabla } _{\frac{\partial }{\partial t}} \frac{\partial }{\partial t} =0,\quad 
\bm{\nabla } _{\frac{\partial }{\partial t}}e_i = \bm{\nabla } _{e_i}\frac{\partial }{\partial t} =0 \quad (1\leq i\leq m) .
\]

First, we can write the formula (\ref{Q_1}) as 
\[
I^{\mathcal{Q}_1}(\varphi ) = \int _M \left\langle \widetilde{\nabla }d\varphi , \widetilde{\nabla }d\varphi \right\rangle d\mu _{g_M} = \int _M \sum _{i,j} \varepsilon_i\varepsilon _j \left\langle \big( \widetilde{\nabla }d\varphi \big) (e_i,e_j) , \big( \widetilde{\nabla }d\varphi \big) (e_i,e_j) \right\rangle d\mu _{g_M} .
\]
For a variation $\{ \varphi _t \} _{t\in I}$ of $\varphi $, it holds that
\begin{align}
\frac{d}{dt} I^{\mathcal{Q}_1}(\varphi _t) &= \frac{d}{dt} \int _M \sum _{i,j} \varepsilon _i\varepsilon _j \left\langle \big( \bm{\widetilde{\nabla }}d\Phi \big) (e_i,e_j) , \big( \bm{\widetilde{ \nabla}} d\Phi \big) (e_i,e_j) \right\rangle d\mu _{g_M} \nonumber \\
&= 2\int _M \sum_{i,j} \varepsilon _i\varepsilon _j \left\langle \bm{\overline{ \nabla }}_{\frac{\partial }{\partial t}}\big( ( \bm{\widetilde{\nabla }} d\Phi ) (e_i,e_j) \big), ( \bm{\widetilde{\nabla }} d\Phi )(e_i, e_j) \right\rangle d\mu _{g_M} .\label{Q_1_t}
\end{align}
Then we have
\begin{align}
\bm{\overline{\nabla }}_{\frac{\partial }{\partial t}} \big( (\bm{\widetilde{\nabla }} d\Phi)(e_i,e_j) \big)
&= \left( \bm{\widetilde{\nabla}} _{\frac{\partial }{\partial t}} \bm{\widetilde{\nabla }}_{e_i}d\Phi \right) (e_j ) \nonumber \\
&= \left( \bm{\widetilde{\nabla }}_{e_i} \bm{\widetilde{\nabla }} _{\frac{\partial }{\partial t}}d\Phi \right) (e_j) - \left( \bm{\widetilde{\nabla }} _{\left[ e_i,\frac{\partial }{\partial t} \right]} d\Phi \right) (e_j) - \left( \widetilde{R}\left( e_i, \frac{\partial }{\partial t} \right) d\Phi \right) (e_j) \nonumber \\
&= \big( \bm{\widetilde{\nabla }} ^2d\Phi \big) \left( e_i,e_j, \frac{\partial }{\partial t} \right) - R^N\left( d\Phi (e_i), d\Phi \left( \frac{\partial }{\partial t} \right) \right) d\Phi (e_j). \label{1st}
\end{align}
By substituting (\ref{1st}) into (\ref{Q_1_t}), we have
\begin{align}
\frac{d}{dt} I^{\mathcal{Q}_1}(\varphi _t)&= 2\int _M \sum _{i,j} \varepsilon _i\varepsilon _j \left\langle \big( \bm{\widetilde{\nabla }} ^2 d\Phi \big) \left( e_i,e_j,\frac{\partial }{\partial t} \right) , \big( \bm{\widetilde{\nabla }} d\Phi \big) (e_i,e_j) \right\rangle d\mu _{g_M} \nonumber \\
&\quad -2\int _M \sum _{i,j} \varepsilon _i\varepsilon _j \left\langle R^N\left(d\Phi (e_i),d\Phi \left( \frac{\partial }{\partial t} \right)\right)d\Phi (e_j) , \big( \bm{\widetilde{\nabla }} d\Phi \big) (e_i,e_j) \right\rangle d\mu _{g_M}\label{Q_1_t_2}.
\end{align}

We need  the following lemma to calculate the first variation of $I^{\mathcal{Q}_1}(\varphi )$.
\begin{lem}
Under the setting above, 
for any variation $\{ \varphi _t \} _{t\in I}$ of $\varphi $, 
it holds
\begin{align}
&\quad \int _M \sum _{i,j} \varepsilon _i \varepsilon _j \left\langle (\bm{\widetilde{\nabla }}^2d\Phi) \left( e_i,e_j,\frac{\partial }{\partial t} \right) , (\bm{\widetilde{\nabla }} d\Phi) (e_i,e_j) \right\rangle d\mu _{g_M} \nonumber \\
&= \int _M \sum _{i,j} \varepsilon _i\varepsilon _j \left\langle d\Phi \left( \frac{\partial }{\partial t} \right) , (\bm{\widetilde{\nabla }} ^3d\Phi) \left( e_i,e_j,e_i,e_j \right) \right\rangle d\mu _{g_M} \label{lemma1}.
\end{align}
\begin{proof}
We define vector fields on $M$ depending on $t\in I$ by
\[
\widetilde{X}_t := \sum _{i,j} \varepsilon _i\varepsilon _j \left\langle (\bm{\widetilde{\nabla }} d\Phi )\left( e_j, \frac{\partial }{\partial t} \right) , (\bm{\widetilde{\nabla }} d\Phi )\left( e_i , e_j \right)  \right\rangle e_i
\]
and
\[
\widetilde{Y}_t := \sum _{i,j} \varepsilon _i\varepsilon _j \left\langle d\Phi \left( \frac{\partial }{\partial t} \right) , \big( \bm{\widetilde{\nabla }} ^2d\Phi \big) \left( e_j , e_i , e_j \right) \right\rangle e_i,
\]
where $\{ e_i \} _{i=1}^m$ is a local pseudo-orthonormal frame field on a neighborhood $U$ of $M$. 
$\widetilde{X}_t$ and $\widetilde{Y}_t$ are well-defined 
because of the independence of the choice of $\{ e_i \}$. 
Hence $\widetilde{X}_t$ and $\widetilde{Y}_t$ are global vector fields on $M$. 

The divergence of $\widetilde{X}_t$ is given by
\begin{align*}
&\quad \mathrm{div} \widetilde{X}_t \\
&= \sum _k \varepsilon _k \left\langle \nabla _{e_k}\widetilde{X}_t , e_k \right\rangle \\
&= \sum _{i,j} \left\{ \varepsilon _i\varepsilon _j \left\langle \bm{\overline{\nabla }} _{e_i} \left( ( \bm{\widetilde{\nabla }} d\Phi) \left( e_j,\frac{\partial }{\partial t} \right) \right) , (\bm{\widetilde{\nabla }} d\Phi) (e_i,e_j) \right\rangle \right. \\
&\quad \left. + \varepsilon _i\varepsilon _j \left\langle ( \bm{\widetilde{\nabla }} d\Phi) \left( e_j,\frac{\partial }{\partial t} \right) , \bm{\overline{\nabla }}_{e_i}\big(  (\bm{\widetilde{\nabla }} d\Phi) (e_i,e_j) \big) \right\rangle \right\} \\
&\quad -\sum_{j,k} \varepsilon _j\varepsilon _k \left\langle (\bm{\widetilde{\nabla }} d\Phi) \left( e_j,\frac{\partial }{\partial t} \right) , (\bm{\widetilde{\nabla }} d\Phi) (\nabla _{e_k}e_k,e_j) \right\rangle \\
&= \sum _{i,j} \varepsilon _i\varepsilon _j \left\{ \left\langle (\bm{\widetilde{\nabla }}^2d\Phi ) \left( e_i,e_j,\frac{\partial }{\partial t} \right) , (\bm{\widetilde{\nabla }} d\Phi)(e_i,e_j) \right\rangle  \right. \\
&\quad \left. + \left\langle ( \bm{\widetilde{\nabla }} d\Phi)\left( \nabla _{e_i}e_j, \frac{\partial }{\partial t} \right) , (\bm{\widetilde{\nabla }} d\Phi)(e_i,e_j) \right\rangle + \left\langle ( \bm{\widetilde{\nabla }} d\Phi) \left( e_j,\frac{\partial }{\partial t} \right) , ( \bm{\widetilde{\nabla }} ^2d\Phi) (e_i,e_i,e_j) \right\rangle \right. \\
&\quad \left. + \left\langle (\bm{\widetilde{\nabla }} d\Phi) \left( e_j,\frac{\partial }{\partial t} \right) , (\bm{\widetilde{\nabla }} d\Phi) \left( e_i, \nabla _{e_i}e_j \right) \right\rangle \right\} .
\end{align*}
At the second equality, we use the following
\begin{align*}
&\quad \sum _{i,j,k} \varepsilon _k\varepsilon _i\varepsilon _j \left\langle (\bm{\widetilde{\nabla }}d\Phi)\left( e_j,\frac{\partial }{\partial t} \right), (\bm{\widetilde{\nabla }}d\Phi) (e_i,e_j) \right\rangle \left\langle \nabla _{e_k}e_i,e_k  \right\rangle \\
&= - \sum _{i,j,k} \varepsilon _k\varepsilon _i\varepsilon _j \left\langle (\bm{\widetilde{\nabla }}d\Phi)\left( e_j,\frac{\partial }{\partial t} \right), (\bm{\widetilde{\nabla }}d\Phi) (e_i,e_j) \right\rangle \left\langle e_i, \nabla _{e_k}e_k  \right\rangle \\
&= -\sum _{j,k} \varepsilon _j\varepsilon _k \left\langle (\bm{\widetilde{\nabla }}d\Phi) \left( e_j,\frac{\partial }{\partial t} \right), (\bm{\widetilde{\nabla }}d\Phi) \left( \nabla _{e_k}e_k,e_j \right) \right\rangle .
\end{align*}

Now, 
take a neighborhood $U$ of $x\in M$ such that the exponential map at $x$ is injective onto $U$, 
which is called a normal neighborhood. 
And we construct a pseudo-orthonormal frame field $\{ e_i \} _{i=1}^m$ 
by parallel transporting a pseudo-orthonormal basis at $x$ 
along a geodesic $\gamma : [0,1]\rightarrow M$ 
from $\gamma (0)=x$ to $\gamma (1)=y$ for every $y\in U$. 
The pseudo-orthonormal frame field $\{ e_i \} _{i=1}^m$ is called a geodesic frame field. 
We note that a geodesic frame field $\{ e_i \}_{i=1}^m$ around a point $x\in M$ satisfies 
\[
\left( \nabla _{e_i}e_j \right) _x =0, \quad \left[ e_i,e_j \right] _x=0 \quad \quad  (1\leq i,j\leq m)
\]
at $x$.
Since $\left( \bm{\nabla } _{e_i}e_j \right) _{(x,t)} = \left( \nabla _{e_i}e_j \right) _x = 0$
for all $t \in I$, we have

\begin{align}
&\quad \big( \mathrm{div}\widetilde{X}_t \big) _x  \nonumber \\
&= \sum _{i,j} \varepsilon _i \varepsilon _j \left\{ \left\langle (\bm{\widetilde{\nabla }}^2d\Phi)_{(x,t)} \left( (e_i)_{(x,t)}, (e_j)_{(x,t)}, \left( \frac{\partial }{\partial t} \right) _{(x,t)} \right), (\bm{\widetilde{\nabla }}d\Phi )_{(x,t)} \left( (e_i)_{(x,t)}, (e_j)_{(x,t)} \right) \right\rangle \right.  \nonumber \\
&\quad \left. + \left\langle (\bm{\widetilde{\nabla }}d\Phi)_{(x,t)} \left( (e_j)_{(x,t)}, \left( \frac{\partial }{\partial t} \right)_{(x,t)}  \right) , (\bm{\widetilde{\nabla }}^2d\Phi)_{(x,t)} \left( (e_i)_{(x,t)} , (e_i)_{(x,t)}, (e_j)_{(x,t)} \right) \right\rangle \right\} . \label{X_x}
\end{align}
Each term of the last formula of (\ref{X_x}) is a tensor, so we have
\begin{align}
\mathrm{div}\widetilde{X}_t &= \sum _{i,j} \varepsilon _i\varepsilon _j \left\{ \left\langle (\bm{\widetilde{\nabla }}^2 d\Phi) \left( e_i,e_j,\frac{\partial }{\partial t} \right),  (\bm{\widetilde{\nabla }} d\Phi)(e_i,e_j) \right\rangle \right. \nonumber \\
&\quad \left. + \left\langle (\bm{\widetilde{\nabla }} d\Phi) \left( e_j,\frac{\partial }{\partial t} \right), (\bm{\widetilde{\nabla }}^2d\Phi) (e_i,e_i,e_j) \right\rangle \right\} , \label{X}
\end{align}
where $\{e_i\}_{i=1}^m$ is an arbitrary local pseudo-orthonormal frame field. 

In a similar way, we calculate the divergence of $\widetilde{Y}_t$.
We have 
\begin{align*}
&\quad \mathrm{div}\widetilde{Y}_t \\
&= \sum _k \varepsilon _k \left\langle \nabla _{e_k}\widetilde{Y}_t , e_k \right\rangle \\
&= \sum _{i,j} \left\{ \varepsilon _i\varepsilon _j \left\langle \bm{\overline{\nabla }} _{e_i}\left( d\Phi \left( \frac{\partial }{\partial t} \right) \right) , (\bm{\widetilde{\nabla }}^2d\Phi) (e_j,e_i,e_j) \right\rangle \right. \\
&\quad \left. + \varepsilon _i\varepsilon _j \left\langle d\Phi \left( \frac{\partial }{\partial t} \right) , \bm{\overline{\nabla }} _{e_i}\big( (\bm{\widetilde{\nabla }}^2d\Phi)(e_j,e_i,e_j) \big) \right\rangle \right\} \\
&\quad -\sum _{j,k} \varepsilon _k \varepsilon _j \left\langle d\Phi \left( \frac{\partial }{\partial t} \right) , (\bm{\widetilde{\nabla }}^2d\Phi) \left( e_j,\nabla _{e_k}e_k,e_j \right) \right\rangle \\
&= \sum_{i,j}\varepsilon _i\varepsilon _j \left\{  \left\langle (\bm{\widetilde{\nabla }} d\Phi) \left( e_i, \frac{\partial }{\partial t} \right) , (\bm{\widetilde{\nabla }}^2d\Phi) (e_j,e_i,e_j) \right\rangle + \left\langle d\Phi \left( \frac{\partial }{\partial t} \right), (\bm{\widetilde{\nabla }}^3d\Phi) (e_i,e_j,e_i,e_j) \right\rangle \right. \\
&\quad \left. + \left\langle d\Phi \left( \frac{\partial }{\partial t} \right) , (\bm{\widetilde{\nabla }}^2d\Phi) \left( \nabla _{e_i}e_j ,e_i,e_j \right) \right\rangle +  \left\langle d\Phi \left( \frac{\partial }{\partial t} \right) , (\bm{\widetilde{\nabla }}^2d\Phi)\left( e_j,e_i,\nabla _{e_i}e_j \right) \right\rangle \right\} .
\end{align*}
Then, assuming that $\{ e_i \}$ is a geodesic frame field around a point $x\in M$, we have
\begin{align}
&\quad \big( \mathrm{div}\widetilde{Y}_t \big) _x \nonumber \\
&= \sum_{i,j} \varepsilon _i\varepsilon _j \left\{ \left\langle  (\bm{\widetilde{\nabla }} d\Phi)_{(x,t)} \left( (e_i)_{(x,t)} , \left( \frac{\partial }{\partial t} \right) _{(x,t)} \right) , (\bm{\widetilde{\nabla }}^2d\Phi)_{(x,t)} \left( (e_j)_{(x,t)}, (e_i)_{(x,t)}, (e_j)_{(x,t)} \right)  \right\rangle \right. \nonumber \\
&\quad \left. + \left\langle (d\Phi )_{(x,t)} \left( \left( \frac{\partial }{\partial t} \right) _{(x,t)} \right) , (\bm{\widetilde{\nabla }} ^3d\Phi) _{(x,t)} \left( (e_i)_{(x,t)} , (e_j)_{(x,t)}, (e_i)_{(x,t)}, (e_j)_{(x,t)}  \right) \right\rangle \right\} \label{Y_x}.
\end{align}
Each term of the right hand side of (\ref{Y_x}) is a tensor, so we have
\begin{align}
\mathrm{div}\widetilde{Y}_t &= \sum _{i,j}\varepsilon _i\varepsilon _j \left\{  \left\langle (\bm{\widetilde{\nabla }} d\Phi) \left( e_i, \frac{\partial }{\partial t} \right) , (\bm{\widetilde{\nabla }}^2d\Phi) (e_j,e_i,e_j) \right\rangle \right. \nonumber \\
&\quad \left. + \left\langle d\Phi \left( \frac{\partial }{\partial t} \right), (\bm{\widetilde{\nabla }} ^3d\Phi) (e_i,e_j,e_i,e_j) \right\rangle \right\} , \label{Y}
\end{align}
where $\{e_i\}_{i=1}^m$ is an arbitrary local pseudo-orthonormal frame field. 

By Green's theorem, we have
\[
\int _M \mathrm{div}\widetilde{X}_t\; d\mu _{g_M} = 0 = \int _M \mathrm{div}\widetilde{Y}_t\; d\mu _{g_M}, 
\]
and together with (\ref{X}) and (\ref{Y}), we have 
\begin{align*}
&\quad \int _M \sum _{i,j} \varepsilon _i\varepsilon _j \left\langle (\bm{\widetilde{\nabla }}^2d\Phi) \left( e_i,e_j,\frac{\partial }{\partial t} \right) , (\bm{\widetilde{\nabla }} d\Phi) (e_i,e_j) \right\rangle d\mu _{g_M} \\
&= \int _M \sum _{i,j} \varepsilon _i\varepsilon _j \left\langle d\Phi \left( \frac{\partial }{\partial t} \right) , (\bm{\widetilde{\nabla }}^3d\Phi) \left( e_i,e_j,e_i,e_j \right) \right\rangle d\mu _{g_M}.
\end{align*}
Here we use the symmetry of $\widetilde{\nabla }^2d\varphi$.
\end{proof}
\end{lem}

Substituting (\ref{lemma1}) into (\ref{Q_1_t_2}), we have
{\small 
\begin{align*}
&\quad \frac{d}{dt} I^{\mathcal{Q}_1}(\varphi _t) \\
&= 2\int _M \sum _{i,j} \varepsilon _i\varepsilon_j \left\langle  (\bm{\widetilde{\nabla }} ^3d\Phi) (e_i,e_j,e_i,e_j) - R^N \left( d\Phi (e_j), (\bm{\widetilde{\nabla }} d\Phi)(e_i,e_j) \right) d\Phi (e_i) , d\Phi \left( \frac{\partial }{\partial t} \right) \right\rangle d\mu _{g_M}.
\end{align*}
}
Therefore we obtain the following theorem. 
\begin{thm}\label{thmQ_1}
Let $(M^m_p,g_M)$ be a compact pseudo-Riemannian manifold, 
$(N^n_q,g_N)$ a pseudo-Riemannian manifold 
and $\varphi : M \rightarrow N$ a $C^{\infty}$-map. 
Consider a $C^{\infty}$-variation $\{ \varphi_t\} _{t\in I}$ of $\varphi$ with variational vector field $V$. 
Then the following formula holds
\begin{align*}
&\quad \left. \frac{d}{dt} I^{\mathcal{Q}_1}(\varphi _t)\right|_{t=0} \\
&= 2\int _M \left\langle \sum _{i,j} \varepsilon _i\varepsilon _j \left\{ \big( \widetilde{\nabla } ^3 d\varphi \big) \left( e_i,e_j,e_i,e_j \right) + R^N \big( ( \widetilde{\nabla } d\varphi ) (e_i,e_j) , d\varphi (e_i) \big) d\varphi (e_j) \right\} , V \right\rangle d\mu _{g_M}, 
\end{align*}
where $\{e_i\} _{i=1}^m$ is a local pseudo-orthonormal frame field of $(M^m_p,g_M)$ with 
$g_M(e_i,e_j) = \varepsilon _i\delta _{ij}$, $\varepsilon _1=\cdots =\varepsilon _p=-1$, 
$\varepsilon _{p+1} = \cdots = \varepsilon _m=1$. 
\end{thm}

For a map $\varphi \in C^{\infty }(M,N)$, 
we define $W_1(\varphi ) \in \Gamma (\varphi ^{-1}TN)$ by
\[
W_1(\varphi ) := \sum _{i,j} \varepsilon _i\varepsilon _j \left\{ \big( \widetilde{\nabla } ^3d\varphi \big) (e_i,e_j,e_i,e_j) + R^N\big( (\widetilde{\nabla } d\varphi )(e_i,e_j), d\varphi (e_i) \big) d\varphi (e_j) \right\} .
\]
Hence $\varphi $ is a $\mathcal{Q}_1$-map if and only if
$W_1(\varphi ) = 0$. 
We can adopt the Euler--Lagrange equation $W_1(\varphi ) = 0$
as the definition of a $\mathcal{Q}_1$-map.
Then the domain $M$ of $\varphi$ is not nesessarily compact.

\subsection{The first variational formula of $\mathcal{Q}_2$-energy}
\label{subsec:Q_2}

In a similar way, 
we show the first variational formula of the $\mathcal{Q}_2$-energy. 
Let $\{ \varphi_t \} _{t\in I}$ be a $C^{\infty}$-variation of $\varphi$ with variational vector field $V$
and $\{ e_i \}$ a local pseudo-orthonormal frame field on a neighborhood $U$. 

First, we can write (\ref{Q_2}) as
\begin{align*}
I^{\mathcal{Q}_2}(\varphi )
&= \int_M \left\langle \mathrm{tr}_{g_M} (\widetilde{\nabla }d\varphi ) , \mathrm{tr}_{g_M} (\widetilde{\nabla }d\varphi ) \right\rangle d\mu _{g_M} \\
&= \int _M\sum _{i,j} \varepsilon _i\varepsilon _j \left\langle (\widetilde{\nabla } d\varphi )(e_i,e_i),(\widetilde{\nabla } d\varphi)(e_j,e_j) \right\rangle d\mu _{g_M}.
\end{align*}
For a variation $\{ \varphi _t \} _{t\in I}$ of $\varphi $, it holds that
\begin{align}
\frac{d}{dt} I^{\mathcal{Q}_2}(\varphi _t) &= \frac{d}{dt} \int _M \sum _{i,j} \varepsilon _i\varepsilon _j \left\langle (\bm{\widetilde{\nabla }} d\Phi )(e_i,e_i) , (\bm{\widetilde{\nabla }} d\Phi)(e_j,e_j) \right\rangle d\mu _{g_M} \nonumber \\
&= 2\int _M \sum _{i,j} \varepsilon _i\varepsilon _j \left\langle \bm{\overline{\nabla }} _{\frac{\partial }{\partial t}} \big( (\bm{\widetilde{\nabla }} d\Phi) (e_i,e_i) \big) , (\bm{\widetilde{\nabla }} d\Phi)(e_j,e_j) \right\rangle d\mu_{g_M} \label{Q_2_t}.
\end{align}
Then we have
\begin{align}
\bm{\overline{\nabla }} _{\frac{\partial }{\partial t}} \big( (\bm{\widetilde{\nabla }} d\Phi) (e_i,e_i) \big)
&= \left( \bm{\widetilde{\nabla }} _{\frac{\partial }{\partial t}}\bm{\widetilde{\nabla }} _{e_i}d\Phi \right)(e_i) \nonumber \\
&= \left( \bm{\widetilde{\nabla}}_{e_i} \bm{\widetilde{\nabla}}_{\frac{\partial}{\partial t}}d\Phi \right) (e_i) - \left( \bm{\widetilde{\nabla }} _{\left[ e_i,\frac{\partial }{\partial t} \right]} d\Phi \right) (e_i) - \left( \widetilde{R}\left( e_i,\frac{\partial }{\partial t} \right) d\Phi \right) (e_i) \nonumber \\
&= (\bm{\widetilde{\nabla }} ^2d\Phi) \left( e_i,e_i,\frac{\partial }{\partial t} \right) - R^N\left( d\Phi (e_i), d\Phi \left( \frac{\partial }{\partial t} \right) \right) d\Phi (e_i) \label{1st_2} .
\end{align}
By substituting (\ref{1st_2}) into (\ref{Q_2_t}), we have
\begin{align}
\frac{d}{dt} I^{\mathcal{Q}_2}(\varphi _t) &= 2\int _M \sum _{i,j} \varepsilon _i\varepsilon _j \left\langle (\bm{\widetilde{\nabla }}^2d\Phi) \left( e_i,e_i, \frac{\partial }{\partial t} \right) , (\bm{\widetilde{\nabla }} d\Phi) (e_j,e_j) \right\rangle d\mu _{g_M} \nonumber \\
&\quad -2\int _M \sum _{i,j} \varepsilon _i\varepsilon _j \left\langle R^N\left( d\Phi (e_i), d\Phi \left( \frac{\partial }{\partial t} \right) \right) d\Phi (e_i) , (\bm{\widetilde{\nabla }} d\Phi) (e_j,e_j) \right\rangle d\mu _{g_M} \label{Q_2_t_2}. 
\end{align}

\begin{lem}
Under the setting above, for any variation $\{ \varphi _t \} _{t\in I}$ of $\varphi $, 
it holds 
\begin{align}
&\quad \int _M \sum _{i,j} \varepsilon _i\varepsilon _j \left\langle (\bm{\widetilde{\nabla }}^2d\Phi) \left( e_i,e_i, \frac{\partial }{\partial t} \right) , (\bm{\widetilde{\nabla }} d\Phi) (e_j,e_j) \right\rangle d\mu _{g_M} \nonumber \\
&= \int _M \sum _{i,j} \varepsilon _i\varepsilon _j \left\langle d\Phi \left( \frac{\partial }{\partial t} \right) , (\bm{\widetilde{\nabla }} ^3d\Phi) (e_i,e_i,e_j,e_j) \right\rangle d\mu _{g_M} \label{lemma2}.
\end{align}
\begin{proof}
For each $t\in I$, we define vector fields on $M$ by
\[
\widehat{X}_t := \sum _{i,j} \varepsilon _i\varepsilon _j \left\langle (\bm{\widetilde{\nabla }} d\Phi )\left( e_i, \frac{\partial }{\partial t} \right) , (\bm{\widetilde{\nabla }} d\Phi )\left( e_j , e_j \right)  \right\rangle e_i
\]
and
\[
\widehat{Y}_t := \sum _{i,j} \varepsilon _i\varepsilon _j \left\langle d\Phi \left( \frac{\partial }{\partial t} \right) , \big( \bm{\widetilde{\nabla }} ^2d\Phi \big) \left( e_i , e_j , e_j \right) \right\rangle e_i,
\]
where $\{ e_i \}_{i=1}^m$ is a pseudo-orthonormal frame field on a neighborhood $U$ of $M$. 
Note that
$\widehat{X}_t$ and $\widehat{Y}_t$ are globally defined vector fields on $M$. 

The divergence of $\widehat{X}_t$ is given by
\begin{align*}
\mathrm{div}\widehat{X}_t
&= \sum_k \varepsilon _k \left\langle \nabla _{e_k} \widehat{X}_t , e_k \right\rangle \\
&= \sum _{i,j}\varepsilon _i\varepsilon _j \left\{ \left\langle (\bm{\widetilde{\nabla }} ^2d\Phi) \left( e_i,e_i,\frac{\partial }{\partial t} \right) , (\bm{\widetilde{\nabla }} d\Phi)(e_j,e_j) \right\rangle \right. \\
&\quad + \left\langle (\bm{\widetilde{\nabla }} d\Phi)\left( e_i,\frac{\partial }{\partial t} \right) , (\bm{\widetilde{\nabla }}^2d\Phi) (e_i,e_j,e_j) \right\rangle \\
&\quad \left. + 2 \left\langle (\bm{\widetilde{\nabla }} d\Phi)\left( e_i,\frac{\partial }{\partial t} \right) , (\bm{\widetilde{\nabla }} d\Phi)\left( \nabla _{e_i}e_j,e_j \right) \right\rangle \right\} .
\end{align*}
Then, assuming that $\{ e_i \}$ is a geodesic frame field around a point $x\in M$, we have
\begin{align}
&\quad \big( \mathrm{div} \widehat{X}_t \big) _x \nonumber \\
&= \sum_{i,j}\varepsilon _i\varepsilon _j \left\{  \left\langle (\bm{\widetilde{\nabla }}^2d\Phi) \left( (e_i)_{(x,t)}, (e_i)_{(x,t)}, \left( \frac{\partial }{\partial t} \right) _{(x,t)} \right) , (\bm{\widetilde{\nabla }}d\Phi) \left( (e_j)_{(x,t)}, (e_j)_{(x,t)} \right) \right\rangle \right. \nonumber \\
& \quad \left. + \left\langle (\bm{\widetilde{\nabla }}d\Phi ) \left( (e_i)_{(x,t)}, \left( \frac{\partial }{\partial t} \right)_{(x,t)} \right) , (\bm{\widetilde{\nabla }}^2d\Phi) _{(x,t)} \left( (e_i)_{(x,t)}, (e_j)_{(x,t)}, (e_j)_{(x,t)} \right) \right\rangle \right\} . \label{X_t_1}
\end{align}
Each term of the right hand side of  (\ref{X_t_1}) is a tensor, so we have
\begin{align}
\mathrm{div}\widehat{X}_t &= \sum _{i,j}\varepsilon _i\varepsilon_j \left\{ \left\langle (\bm{\widetilde{\nabla }} ^2d\Phi) \left( e_i,e_i,\frac{\partial }{\partial t} \right) , (\bm{\widetilde{\nabla }} d\Phi )(e_j,e_j) \right\rangle \right. \nonumber \\
&\quad \left. + \left\langle (\bm{\widetilde{\nabla }} d\Phi)\left( e_i,\frac{\partial }{\partial t} \right) , (\bm{\widetilde{\nabla }}^2d\Phi) (e_i,e_j,e_j) \right\rangle \right\} , \label{X_2}
\end{align}
where $\{e_i\}_{i=1}^m$ is an arbitrary local pseudo-orthonormal frame field. 

In a similar way, we calculate the divergence of $\widehat{Y}_t$. We have
\begin{align*}
\mathrm{div}\widehat{Y}_t
&= \sum _k \varepsilon _k \left\langle \nabla _{e_k}\widehat{Y}_t , e_k \right\rangle \\
&= \sum _{i,j}\varepsilon _i\varepsilon _j \left\{  \left\langle (\bm{\widetilde{\nabla }} d\Phi) \left( e_i, \frac{\partial }{\partial t} \right) , (\bm{\widetilde{\nabla }}^2d\Phi)(e_i,e_j,e_j) \right\rangle \right. \\
&\quad + \left\langle d\Phi \left( \frac{\partial }{\partial t} \right) , (\bm{\widetilde{\nabla }} ^3d\Phi) (e_i,e_i,e_j,e_j) \right\rangle \\
&\quad \left. + 2 \left\langle d\Phi \left( \frac{\partial }{\partial t} \right) , (\bm{\widetilde{\nabla }} ^2d\Phi) \left( e_i,\nabla _{e_i}e_j,e_j \right) \right\rangle \right\} .
\end{align*}
Then, assuming that $\{ e_i \}$ is a geodesic frame field around a point $x\in M$, we have
\begin{align}
&\quad \big( \mathrm{div}\widehat{Y}_t \big) _x \nonumber \\
&= \sum _{i,j} \varepsilon _i\varepsilon _j \left\{ \left\langle  (\bm{\widetilde{\nabla }} d\Phi) _{(x,t)} \left( (e_i)_{(x,t)}, \left( \frac{\partial }{\partial t}\right) _{(x,t)} \right) , (\bm{\widetilde{\nabla }}^2d\Phi) _{(x,t)} \left( (e_i)_{(x,t)}, (e_j)_{(x,t)}, (e_j)_{(x,t)} \right) \right\rangle \right. \nonumber \\
&\quad \left. + \left\langle (d\Phi )_{(x,t)} \left( \left( \frac{\partial }{\partial t} \right) _{(x,t)} \right) , (\bm{\widetilde{\nabla }} ^3d\Phi) _{(x,t)} \left( (e_i)_{(x,t)}, (e_i)_{(x,t)}, (e_j)_{(x,t)}, (e_j)_{(x,t)} \right) \right\rangle \right\} . \label{Y_t_2}
\end{align}
Each term of the right hand side of (\ref{Y_t_2}) is a tensor, so we have
\begin{align}
\mathrm{div}\widehat{Y}_t &= \sum _{i,j}\varepsilon _i\varepsilon _j \left\{ \left\langle (\bm{\widetilde{\nabla }} d\Phi) \left( e_i,\frac{\partial }{\partial t} \right) , (\bm{\widetilde{\nabla }}^2d\Phi)(e_i,e_j,e_j) \right\rangle \right. \nonumber \\
&\quad \left. + \left\langle d\Phi \left( \frac{\partial }{\partial t} \right) , (\bm{\widetilde{\nabla }} ^3d\Phi) (e_i,e_i,e_j,e_j) \right\rangle \right\} ,  \label{Y_2}
\end{align}
where $\{e_i\}_{i=1}^m$ is an arbitrary local pseudo-orthonormal frame field. 

By Green's theorem, we have
\[
\int _M \mathrm{div}\widehat{X}_t\; d\mu _{g_M} = 0 = \int _M \mathrm{div}\widehat{Y}_t\; d\mu _{g_M},
\]
and together with (\ref{X_2}) and (\ref{Y_2}), we have
\begin{align*}
&\quad \int _M \sum _{i,j} \varepsilon _i\varepsilon _j \left\langle (\bm{\widetilde{\nabla }}^2d\Phi) \left( e_i,e_i, \frac{\partial }{\partial t} \right) , (\bm{\widetilde{\nabla }} d\Phi) (e_j,e_j) \right\rangle d\mu _{g_M} \\
&= \int _M \sum _{i,j} \varepsilon _i\varepsilon _j \left\langle d\Phi \left( \frac{\partial }{\partial t} \right) , (\bm{\widetilde{\nabla }} ^3d\Phi) (e_i,e_i,e_j,e_j) \right\rangle d\mu _{g_M}.
\end{align*}
\end{proof}
\end{lem}

Substituting (\ref{lemma2}) into (\ref{Q_2_t_2}), we have
{\small 
\begin{align*}
&\quad \frac{d}{dt} I^{\mathcal{Q}_2}(\varphi _t) \\
&= 2\int _M \sum _{i,j} \varepsilon _i\varepsilon _j \left\langle  (\bm{\widetilde{\nabla }} ^3d\Phi) (e_i,e_i,e_j,e_j) - R^N \left( d\Phi (e_i), (\bm{\widetilde{\nabla }} d\Phi)(e_j,e_j) \right) d\Phi (e_i) , d\Phi \left( \frac{\partial }{\partial t} \right) \right\rangle d\mu _{g_M}.
\end{align*}
}
Therefore we obtain the following theorem. 
\begin{thm}\label{thmQ_2}
Let $(M^m_p,g_M)$ be a compact pseudo-Riemannian manifold, 
$(N^n_q,g_N)$ a pseudo-Riemannian manifold 
and $\varphi : M \rightarrow N$ a $C^{\infty}$-map. 
Consider a $C^{\infty}$-variation $\{\varphi_t\} _{t\in I}$ of $\varphi$ with variational vector field $V$.
Then the following formula holds
\begin{align*}
&\quad \left. \frac{d}{dt} I^{\mathcal{Q}_2}(\varphi _t)\right|_{t=0} \\
&= 2\int _M \left\langle \sum _{i,j} \varepsilon _i\varepsilon _j \left\{ \big( \widetilde{\nabla  }^3 d\varphi \big) \left( e_i,e_i,e_j,e_j \right) + R^N \big( ( \widetilde{ \nabla } d\varphi ) (e_i,e_i), d\varphi (e_j) \big) d\varphi (e_j) \right\} ,V \right\rangle d\mu _{g_M},
\end{align*}
where $\{e_i\} _{=1}^m$ is a local pseudo-orthonormal frame field of $(M^m_p,g_M)$ with 
$g_M(e_i,e_j) = \varepsilon _i\delta _{ij}$, $\varepsilon _1=\cdots =\varepsilon _p=-1$, 
$\varepsilon _{p+1} = \cdots = \varepsilon _m=1$. 
\end{thm}

For a map $\varphi \in C^{\infty }(M,N)$,
we define $W_2(\varphi ) \in \Gamma (\varphi ^{-1}TN)$ by
\[
W_2(\varphi ) := \sum _{i,j} \varepsilon _i\varepsilon _j \left\{ \big( \widetilde{\nabla } ^3 d\varphi \big) \left( e_i,e_i,e_j,e_j \right) + R^N \big( ( \widetilde{\nabla } d\varphi ) (e_i,e_i), d\varphi (e_j) \big) d\varphi (e_j) \right\} .
\]
Hence $\varphi $ is a $\mathcal{Q}_2$-map if and only if
$W_2(\varphi) = 0$.

\begin{rem}
For a pseudo-Riemannian manifold $(M^m_p,g_M)$, 
if the index $p=0$ then $(M^m_0,g_M)$ is a Riemannian manifold. 
Therefore a map $\varphi : (M^m_p,g_M) \rightarrow (N^n_q,g_N)$
between Riemannian manifolds is a $\mathcal{Q}_1$-map if and only if
\[
\sum _{i,j} \left\{ \big( \widetilde{\nabla } ^3 d\varphi \big) \left( e_i,e_j,e_i,e_j \right) + R^N \big( ( \widetilde{\nabla } d\varphi ) (e_i,e_j), d\varphi (e_i),  \big) d\varphi (e_j) \right\} = 0,
\]
where $\{e_i\} _{i=1}^m$ is a local orthonormal frame field of $(M^m,g_M)$. 
Similarly, we have that a map $\varphi : (M^m_p,g_M) \rightarrow (N^n_q,g_N)$ 
between Riemannian manifolds is a $\mathcal{Q}_2$-map if and only if
\[
\sum _{i,j} \left\{ \big( \widetilde{\nabla } ^3 d\varphi \big) \left( e_i,e_i,e_j,e_j \right) + R^N \big( ( \widetilde{\nabla } d\varphi ) (e_i,e_i), d\varphi (e_j),  \big) d\varphi (e_j) \right\} = 0.
\]
\end{rem}

By Theorem \ref{thmQ_1} and Theorem \ref{thmQ_2}, 
we obtain all the first variational formulae of the integral invariants 
which belong to the space spanned by the $\mathcal{Q}_1$-energy and $\mathcal{Q}_2$-energy. 

By comparing the first variational formula of the bienergy (c.f. \cite{Jiang})
and that of $\mathcal{Q}_2$-energy (Theorem~\ref{thmQ_2}),
we have the following proposition. 
\begin{prop}
Let $\varphi : M \rightarrow N$ be a $C^{\infty}$-map 
between pseudo-Riemannian manifolds $(M^m_p,g_M)$ and $(N^n_q,g_N)$. 
Then the following formula holds
\[
-\overline{\nabla }^*\overline{\nabla }\tau (\varphi ) = \sum _{i,j}\varepsilon _i\varepsilon _j (\widetilde{\nabla }^3d\varphi )(e_i,e_i,e_j,e_j),
\]
where $-\overline{\nabla }^*\overline{\nabla }$ is the rough Laplacian and 
$\{e_i\}_{i=1}^m$ is a local pseudo-orthonormal frame field of $(M^m_p,g_M)$.
\begin{proof}
For any $V\in \Gamma (\varphi ^{-1}TN)$, we define vector fields on $M$ by
\[
W := \sum _{i,j}^m \varepsilon _i\varepsilon _j \left\langle V, \overline{\nabla }_{e_i}\big( (\widetilde{\nabla }d\varphi )(e_j,e_j) \big) \right\rangle e_i 
\]
and
\[
W':= \sum _{i,j}^m \varepsilon _i\varepsilon _j \left\langle V, (\widetilde{\nabla }^2d\varphi )(e_i,e_j,e_j) \right\rangle e_i, 
\]
where $\{e_i\}_{i=1}^m$ is a local pseudo-orthonormal frame field of $(M^m_p,g_M)$.
Then, assuming that $\{e_i\}$ is a geodesic frame field around a point $x\in M$, we have
\begin{align*}
W'_x &= \sum _{i,j} \varepsilon_i\varepsilon _j \left\langle V_x, \big( \overline{\nabla }_{e_i}\big( (\widetilde{\nabla }d\varphi )(e_j,e_j) \big) \big) _x- 2\big( (\widetilde{\nabla }d\varphi )(\nabla _{e_i}e_j,e_j) \big)_x \right\rangle (e_i)_x \\
&= \sum _{i,j}\varepsilon _i\varepsilon _j \left\langle V_x, \big( \overline{\nabla }_{e_i} \big( (\widetilde{\nabla }d\varphi )(e_j,e_j) \big) \big) _x \right\rangle (e_i)_x \\
&= W_x.
\end{align*}
Therefore $W=W'$. Thus, 
\begin{align*}
0 &= \mathrm{div}(W-W')_x \\
&= \sum _{i,j} \varepsilon _i\varepsilon_j \left\{ \left\langle \big( \overline{\nabla }_{e_i}V \big) _x, \big( \overline{\nabla }_{e_i}\big( (\widetilde{\nabla }d\varphi)(e_j,e_j) \big) \big) _x \right\rangle  + \left\langle V_x, \big( \overline{\nabla }_{e_i} \big( \overline{\nabla }_{e_i} \big( (\widetilde{\nabla }d\varphi) (e_j,e_j) \big) \big) \big)_x \right\rangle \right\} \\
&\quad - \sum _{i,j} \varepsilon _i\varepsilon _j \left\{ \left\langle \big( \overline{\nabla }_{e_i}V \big) _x, \big( \big( \widetilde{\nabla }^2d\varphi \big) (e_i,e_j,e_j) \big)_x \right\rangle + \left\langle V_x, \big( \overline{\nabla }_{e_i}\big( (\widetilde{\nabla }^2d\varphi )(e_i,e_j,e_j) \big) \big) _x \right\rangle \right\} \\
&= \left\langle V_x, \big( -\overline{\nabla }^*\overline{\nabla}\tau (\varphi ) \big) _x - \sum _{i,j} \varepsilon _i\varepsilon _j \big( (\widetilde{\nabla }^3d\varphi)(e_i,e_i,e_j,e_j) \big) _x \right\rangle ,
\end{align*}
where $\{e_i\}_{i=1}^m$ is a geodesic frame field around a point $x\in M$. 
So we have
\[
-\overline{\nabla}^*\overline{\nabla }\tau (\varphi ) = \sum _{i,j} \varepsilon _i\varepsilon _j (\widetilde{\nabla }^3d\varphi )(e_i,e_i,e_j,e_j),
\]
where $\{e_i\} _{i=1}^m$ is an arbitrary local pseudo-orthonormal frame field.
\end{proof}
\end{prop}

\section{The Euler--Lagrange equation of the Chern--Federer energy}
\label{sec:CF maps}

We inherit the settings in the previous section.
In this section,
we introduce the Chern--Federer energy functional for a map $\varphi : (M^m_p,g_M) \rightarrow (N^n_q,g_N)$
between pseudo-Riemannian manifolds,
which is an integral invariant defined by a homogeneous polynomial of degree two
on $\mathrm{II} (\mathbb{E}^m_p,\mathbb{E}^n_q)$ called the Chern--Federer polynomial.
Then we verify the Euler--Lagrange equation of the Chern--Federer energy functional. 

For $H=(h_{ij}^\alpha )\in \mathrm{II} (\mathbb{E}^m_p,\mathbb{E}^n_q)$, 
the \textit{Chern--Federer polynomial} $\mathrm{CF}(H)$ is defined by
\begin{equation} \label{eq:CF polynomial}
\mathrm{CF}(H) := \mathcal{Q}_2(H) - \mathcal{Q}_1(H).
\end{equation}
From Theorems~\ref{thmQ_1} and \ref{thmQ_2},
the Euler--Lagrange equation of the \textit{Chern--Federer energy functional} $I^{\mathrm{CF}}(\varphi)$ is
\begin{align}
0 &= W_2(\varphi ) - W_1(\varphi ) \nonumber \\
&= \sum _{i,j} \varepsilon _i\varepsilon _j \left\{ (\widetilde{\nabla } ^3d\varphi)(e_i,e_i,e_j,e_j) - (\widetilde{\nabla } ^3d\varphi)(e_i,e_j,e_i,e_j) \right. \nonumber \\
&\quad \left. + R^N\big( (\widetilde{\nabla } d\varphi )(e_i,e_i) ,d\varphi (e_j) \big) d\varphi (e_j) - R^N\big( (\widetilde{\nabla } d\varphi )(e_i,e_j) , d\varphi (e_i) \big) d\varphi (e_j)  \right\} , \label{EL_CF}
\end{align}
where $\{e_i\} _{i=1}^m$ is a local pseudo-orthonormal frame field of $(M^m_p,g_M)$. 
In this section, we give alternative expressions of the Euler--Lagrange equation of the Chern--Federer energy functional. 
In particular, 
the Euler--Lagrange equation of $I^{\mathrm{CF}}(\varphi)$ is a second-order partial differential equation for $\varphi$. 
Moreover, we describe the symmetry of the Euler--Lagrange equation of the Chern--Federer energy functional and that of the Chern--Federer polynomial. 

We also introduce the Willmore--Chen energy functional,
which is an integral invariant defined by the homogeneous polynomial of degree two called the Willmore--Chen polynomial.
For $H=(h_{ij}^\alpha )\in \mathrm{II}(\mathbb{E}^m_p,\mathbb{E}^n_q)$, 
the \textit{Willmore--Chen polynomial} $\mathrm{WC}(H)$ is defined by
\[
\mathrm{WC}(H) := m\mathcal{Q}_1(H) - \mathcal{Q}_2(H).
\]

Let $\alpha$ and $\beta$ be constant numbers such that $\alpha^{2} + \beta^{2} \neq 0$. 
A $C^{\infty}$-map $\varphi : M \rightarrow N$ is called an \textit{$(\alpha \mathcal{Q}_{1} + \beta \mathcal{Q}_{2})$-map}
if it satisfies
\begin{equation*}
\alpha W_{1}(\varphi) + \beta W_{2}(\varphi) = 0.
\end{equation*}
By definition, an $(\alpha \mathcal{Q}_{1} + \beta \mathcal{Q}_{2})$-map $\varphi$ is 
\begin{itemize} 
\item a $\mathcal{Q}_{1}$-map when $(\alpha, \beta) = (1, 0)$; 
\item a $\mathcal{Q}_{2}$-map, that is, a biharmonic map, when $(\alpha, \beta) = (0, 1)$; 
\item a \textit{Chern--Federer map} when $(\alpha, \beta) = (-1, 1)$; 
\item a \textit{Willmore--Chen map} when $(\alpha, \beta) = (m, -1)$. 
\end{itemize}
In Section~\ref{sec:CF submanifolds}, we construct some examples of these maps. 

\subsection{Alternative expression of the Euler--Lagrange equation of the Chern--Federer energy functional I}
\label{subsec:alternative expression 1}

First, we prepare the following lemmas. 
 \begin{lem}\label{3rd_fundamental}
For a smooth map $\varphi : M \rightarrow N$ and 
$X,Y,Z\in \Gamma (TM)$, the following equation holds:
\begin{align*}
(\widetilde{\nabla }^2d\varphi ) (X,Y,Z) - (\widetilde{\nabla }^2d\varphi )(Y,X,Z)
= R^N\left( d\varphi (X),d\varphi (Y) \right) d\varphi (Z) - d\varphi \left( R^M(X,Y)Z \right) .
\end{align*}
\begin{proof}
Let $\{ e_i \} _{i=1}^m$ be a geodesic frame field of $(M^m_p,g_M)$ around $x\in M$. 
At $x$, we have
\begin{align*}
&\quad (\widetilde{\nabla }^2d\varphi )(e_i,e_j,e_k) - (\widetilde{\nabla }^2d\varphi )(e_j,e_i,e_k) \\
&= \overline{\nabla }_{e_i} \left( \overline{\nabla }_{e_j} \left( d\varphi (e_k) \right) - d\varphi (\nabla _{e_j}e_k) \right) - \overline{\nabla }_{e_j} \left( \overline{\nabla }_{e_i}\left( d\varphi (e_k)\right)  - d\varphi (\nabla _{e_i}e_k) \right) \\
&= \left( \overline{\nabla }_{e_i}\overline{\nabla }_{e_j} - \overline{\nabla }_{e_j}\overline{\nabla }_{e_i} - \overline{\nabla }_{[e_i,e_j]} \right) d\varphi (e_k) \\
&\quad -\left\{ (\widetilde{\nabla }d\varphi )\left( e_i, \nabla _{e_j}e_k \right) + d\varphi \left( \nabla _{e_i}\nabla _{e_j}e_k  \right) \right\} + \left\{ (\widetilde{\nabla }d\varphi ) (e_j,\nabla _{e_i}e_k) + d\varphi \left( \nabla _{e_j}\nabla _{e_i}e_k \right) \right\} \\
&= R^{\varphi ^{-1}TN} (e_i,e_j)d\varphi (e_k) - d\varphi \left( \nabla _{e_i}\nabla _{e_j}e_k - \nabla _{e_j}\nabla _{e_i}e_k - \nabla _{[e_i,e_j]}e_k \right) \\
&= R^N\left( d\varphi (e_i),d\varphi (e_j) \right) d\varphi (e_k) - d\varphi \left( R^M(e_i,e_j)e_k \right) .
\end{align*}
Since all terms of the first and last formulae are tensors, we have the lemma.  
\end{proof}
\end{lem}

\begin{lem}\label{4th_fundamental}
For a smooth map $\varphi : M \rightarrow N$ and 
$X,Y,Z,W\in \Gamma (TM)$, the following equation holds:
\begin{align*}
&\quad  (\widetilde{\nabla }^3d\varphi )(X,Y,Z,W) - (\widetilde{\nabla }^3d\varphi)(X,Z,Y,W) \\
&= \left( \nabla R^N \right) \left( d\varphi (X), d\varphi (Y), d\varphi (Z) \right) d\varphi (W) + R^N \big( (\widetilde{\nabla }d\varphi)(X,Y), d\varphi (Z) \big) d\varphi (W)  \\
&\quad + R^N \big( d\varphi (Y) , (\widetilde{\nabla }d\varphi)(X,Z) \big) d\varphi (W)+ R^N\left( d\varphi (Y),d\varphi (Z) \right) (\widetilde{\nabla }d\varphi) (X,W) \\
&\quad - (\widetilde{\nabla }d\varphi )\left( X,R^M(Y,Z)W \right)  -d\varphi \left( \left( \nabla R^M \right) (X,Y,Z) W \right). 
\end{align*}
\begin{proof}
First, we show that the following equation:
\begin{align}
&\quad (\widetilde{\nabla }^3d\varphi )(X,Y,Z,W) - (\widetilde{\nabla }^3d\varphi )(X,Z,Y,W) \nonumber \\
&= \big( \nabla R^{\varphi ^{-1}TN} \big) (X,Y,Z) d\varphi (W) + R^{\varphi ^{-1}TN} (Y,Z) (\widetilde{\nabla }d\varphi )(X,W) \nonumber \\
&\quad - (\widetilde{\nabla }d\varphi )(X, R^M(Y,Z)W) - d\varphi \left( \left( \nabla R^M \right) (X,Y,Z)W \right) \label{CF_first},
\end{align}
where $X,Y,Z,W\in \Gamma (TM)$. 
Let $\{ e_i \} _{i=1}^m$ be a geodesic frame field of $(M^m_p,g_M)$ around $x\in M$. 
At $x$, we have
\begin{align*}
&\quad (\widetilde{\nabla }^3d\varphi )(e_i,e_j,e_k,e_l) - (\widetilde{\nabla }^3d\varphi )(e_i,e_k,e_j,e_l) \\
&= \overline{\nabla }_{e_i}\left( R^N\left( d\varphi (e_j),d\varphi (e_k) \right) d\varphi (e_l) \right) - (\widetilde{\nabla }d\varphi ) \left( e_i, R^M(e_j,e_k)e_l \right) - d\varphi \left( \nabla _{e_i} \left( R^M(e_j,e_k)e_l \right) \right) \\
&= \big( \nabla R^{\varphi ^{-1}TN} \big) (e_i,e_j,e_k)d\varphi (e_l) + R^{\varphi ^{-1}TN} (e_j,e_k)(\widetilde{\nabla }d\varphi )(e_i,e_l) - (\widetilde{\nabla }d\varphi )\left( e_i, R^M(e_j,e_k)e_l \right) \\
&\quad -d\varphi \left( \big( \nabla R^M \big) (e_i,e_j,e_k) e_l \right) .
\end{align*}
Here, the first equality holds because of Lemma~\ref{3rd_fundamental}. 
Since all terms of the first and last formulae are tensors, we have (\ref{CF_first}). 
Then, on $\mathrm{End}\varphi ^{-1}TN$, we have
\begin{align}
R^{\varphi ^{-1}TN}(Y,Z) \big((\widetilde{\nabla }d\varphi )(X,W)\big)
= R^N(d\varphi(Y), d\varphi(Z)) \big((\widetilde{\nabla }d\varphi )(X,W)\big), \label{CF_2nd}
\end{align}
where $X,Y,Z,W\in \Gamma (TM)$, 
since the following equation holds:
\[
R^{\varphi ^{-1}TN} (X,Y) = (\varphi^{-1}R^N)(X,Y) = R^N(d\varphi (X), d\varphi (Y)).
\]
Also, we can verify the following equation: 
\begin{align}
\big( \nabla R^{\varphi ^{-1}TN} \big) (X,Y,Z) 
&= (\nabla R^N) (d\varphi (X), d\varphi (Y), d\varphi (Z)) \nonumber \\
&\quad + R^N(d\varphi (Y), (\widetilde{\nabla }d\varphi )(X,Z)) + R^N( (\widetilde{\nabla }d\varphi )(X,Y), d\varphi (Z) ). \label{CF_3rd}
\end{align}
Thus we have
{\small 
\begin{align*}
&\quad \big( \nabla R^{\varphi ^{-1}TN} \big) (X,Y,Z) \\
&= \overline{\nabla } _X(R^{\varphi ^{-1}TN} (Y,Z)) - R^{\varphi ^{-1}TN} (\nabla _XY,Z) - R^{\varphi ^{-1}TN} (Y, \nabla _XZ) \\
&= \overline{\nabla }_X(R^N(d\varphi (Y), d\varphi (Z))) - R^N (d\varphi (\nabla _XY), d\varphi (Z)) - R^N(d\varphi (Y), d\varphi (\nabla _XZ)) \\
&= \overline{\nabla }_X (R^N (d\varphi (Y), d\varphi (Z))) + R^N ((\widetilde{\nabla }d\varphi )(X,Y), d\varphi (Z)) + R^N(d\varphi (Y), (\widetilde{\nabla }d\varphi )(X,Z)) \\
&\quad - R^N( \overline{\nabla }_X(d\varphi (Y)), d\varphi (Z) ) - R^N( d\varphi (Y), \overline{\nabla }_X(d\varphi (Z)) ) \\
&= R^N ((\widetilde{\nabla }d\varphi )(X,Y), d\varphi (Z)) + R^N(d\varphi (Y), (\widetilde{\nabla }d\varphi )(X,Z)) + (\nabla R^N) (d\varphi (X), d\varphi (Y), d\varphi (Z)). 
\end{align*}
}
Here, by taking local frame fields of $(M,g_M)$ and $(N,g_N)$ 
and calculating locally, we verify the last equality. 
Therefore the assertion holds from (\ref{CF_first}), (\ref{CF_2nd}) and (\ref{CF_3rd}). 
\end{proof}
\end{lem}
\begin{rem}
Recall that $\nabla R^{\varphi^{-1}TN}$ is the derivative of the curvature tensor field $R^{\varphi ^{-1}TN}$, defined by
\begin{align*}
\big( \nabla R^{\varphi ^{-1}TN} \big) (X,Y,Z) s
&:= \overline{\nabla }_X\big( R^{\varphi ^{-1}TN}(Y,Z)s \big) - R^{\varphi ^{-1}TN}\left( \nabla _XY,Z \right) s \\
&\quad - R^{\varphi ^{-1}TN}\left( Y,\nabla _XZ \right) s - R^{\varphi ^{-1}TN}(Y,Z) \overline{\nabla }_Xs,
\end{align*}
where $X,Y,Z\in \Gamma (TM)$ and $s\in \Gamma (\varphi ^{-1}TN)$. 
\end{rem}

Using Lemma \ref{4th_fundamental}, we obtain the following proposition. 
\begin{prop}\label{CFmap_eq}
A smooth map $\varphi : M \rightarrow N$ is a Chern--Federer map if and only if 
\begin{align}
0 &= \sum _{i,j} \varepsilon _i\varepsilon _j \left\{ \left( \nabla R^N \right) \left( d\varphi (e_i),d\varphi (e_i),d\varphi (e_j)\right) d\varphi (e_j) - (\widetilde{\nabla }d\varphi) \left( e_i, R^M(e_i,e_j)e_j \right) \right. \nonumber \\
&\quad - d\varphi \left( \left( \nabla R^M \right) (e_i,e_i,e_j) e_j \right)  +2R^N\big( (\widetilde{\nabla }d\varphi)(e_i,e_i), d\varphi (e_j) \big) d\varphi (e_j) \nonumber \\
&\quad \left. + 2 R^N\big( d\varphi (e_i), (\widetilde{\nabla }d\varphi )(e_i,e_j) \big) d\varphi (e_j) \right\} , \label{CF_eq}
\end{align}
where $\{ e_i \}_{i=1}^m$ is a local pseudo-orthonormal frame field of $(M^m_p,g_M)$. 
\end{prop}
By the equation (\ref{CF_eq}), 
it can be seen that the Euler--Lagrange equation of the Chern--Federer energy functional for a map $\varphi$ is a second-order partial differential equation for $\varphi$. 

\subsection{Alternative expression of the Euler--Lagrange equation of the Chern--Federer energy functional II}
\label{subsec:alternative expression 2}

Here, we express the Chern--Federer energy functional as follows:
\begin{align*}
I^{\mathrm{CF}}(\varphi ) &= I^{\mathcal{Q}_2-\mathcal{Q}_1}(\varphi ) \\
&= \int _M \left( \sum _\alpha \varepsilon '_\alpha \left( \sum_i \varepsilon _i h_{ii}^\alpha \right)^2 - \sum _\alpha \varepsilon '_\alpha \sum _{i,j} \varepsilon _i \varepsilon _j (h_{ij}^\alpha )^2 \right) d\mu _{g_M} \\ 
&= \int _M \sum _\alpha \varepsilon '_\alpha \sum _{i,j} \varepsilon _i \varepsilon _j \mathrm{det} \begin{pmatrix}
h_{ii}^\alpha & h_{ij}^\alpha \\
h_{ij}^\alpha & h_{jj}^\alpha 
\end{pmatrix} d\mu _{g_M}.
\end{align*}
Then we have the following alternative expression of the Euler--Lagrange equation of the Chern--Federer energy functional. 

\begin{thm}\label{thm:symmetry of EL equation}
Let $\varphi : M \rightarrow N$ be a $C^{\infty}$-map between pseudo-Riemannian manifolds. 
We define $(0,4)$-type tensor fields $\mu$ and $\nu$ valued on $\varphi ^{-1}TN$ by
\[
\mu (X_1, X_2, X_3, X_4) := (\widetilde{\nabla }^3d\varphi )( X_1, X_2, X_3, X_4 ) 
\]
and 
\[
\nu (X_1,X_2,X_3,X_4) := R^N\big( (\widetilde{\nabla }d\varphi )(X_3,X_4 ), d\varphi (X_1) \big) d\varphi (X_2) ,
\]
where $X_1,X_2,X_3,X_4\in \Gamma (TM)$.
Then $\varphi$ is a Chern--Federer map if and only if 
\begin{align}
C(\mu + \nu )=0. \label{EL_CF_2}
\end{align}
Here $C$ is the contraction of a $(0,4)$-tensor field on $M$ defined by
\[
C:=\mathrm{det} \begin{pmatrix}
C_{12} & C_{13} \\
C_{24} & C_{34}
\end{pmatrix}, 
\]
where $C_{ij}$ is the contraction of the $i$-th and $j$-th variables.
\end{thm}

\begin{proof}
From the definition of $\mu$ and $\nu$, 
we have $\mu ,\nu \in \Gamma (T^*M\otimes T^*M\otimes (T^*M\odot T^*M)\otimes \varphi ^{-1}TN)$. 
For simplicity, we set
\[
\mu _{ijkl} := \mu (e_i,e_j,e_k,e_l)
\]
and
\[
\nu _{ijkl} := \nu (e_i,e_j,e_k,e_l),  
\] 
where $\{ e_i \} _{i=1}^m$ is a local pseudo-orthonormal frame field of $M$. 
Note that, by the pseudo-Riemannian metric $g_M$,
there is a natural correspondence between a covariant tensor and a contravariant tensor on $M$.
Hence we can consider a contraction of $(0,4)$-tensor field on $M$.
Then we have
\begin{align*}
&\quad \sum _{i,j}\varepsilon _i\varepsilon _j \left\{ (\widetilde{\nabla }^3d\varphi ) (e_i,e_i,e_j,e_j) - (\widetilde{\nabla }^3d\varphi )(e_i,e_j,e_i,e_j) \right\} \\
&= \sum _{i,j} \varepsilon _i\varepsilon _j\left( \mu _{iijj}-\mu _{ijij} \right) = \sum _{i,j}\left( \mu _i\;^i\;_j\;^j - \mu _{ij} \;^{ij} \right) \\
&= C_{12}C_{34}\mu - C_{13}C_{24}\mu = \mathrm{det} \begin{pmatrix}
C_{12} & C_{13} \\
C_{24} & C_{34}
\end{pmatrix} \mu. 
\end{align*}
In a similar way, we have
\begin{align*}
&\quad \sum _{i,j}\varepsilon _i\varepsilon _j \left\{ R^N\big( (\widetilde{\nabla }d\varphi )(e_i,e_i), d\varphi (e_j) \big) d\varphi (e_j) - R^N\big( (\widetilde{\nabla }d\varphi )(e_i,e_j),d\varphi (e_i) \big) d\varphi (e_j) \right\} \\
&= \mathrm{det}\begin{pmatrix}
C_{12} & C_{13} \\
C_{24} & C_{34}
\end{pmatrix} \nu .
\end{align*}
Therefore the Euler--Lagrange equation (\ref{EL_CF}) of the Chern--Federer energy functional can be expressed as the following equation:
\[
\mathrm{det}\begin{pmatrix}
C_{12} & C_{13} \\
C_{24} & C_{34}
\end{pmatrix}(\mu + \nu )=0.
\]
\end{proof}

In addition, 
we observe symmetry of the equation (\ref{EL_CF_2}) and the Chern--Federer polynomial (\ref{eq:CF polynomial}).
Let $\mathcal{U}$ be the space of $O(p,m-p)\times O(q,n-q)$-invariant homogeneous polynomials of degree two
on $\mathrm{II}(\mathbb{E}^m_p,\mathbb{E}^n_q)$,
which is spanned by $\mathcal{Q}_1$ and $\mathcal{Q}_2$:
\[
\mathcal{U}:= \mathrm{span}_{\mathbb{R}}\left\{ \mathcal{Q}_1, \mathcal{Q}_2 \right\}.
\]
Also we denote by $\mathcal{V}$ the space of sections of $\varphi^{-1}TN$
spanned by $v_1:=C_{13}C_{24}(\mu +\nu )$ and $v_2:=C_{12}C_{34}(\mu +\nu)$:
\[
\mathcal{V}:= \mathrm{span}_{\mathbb{R}}\left\{ v_1, v_2 \right\}.
\]
Then, 
by the first variational formula of the $(\alpha \mathcal{Q}_1 + \beta \mathcal{Q}_2)$-energy functional,
we have a linear isomorphism between $\mathcal{U}$ and $\mathcal{V}$.
From the first variational formula (\ref{EL_CF_2}) of the Chern--Federer energy functional,
we observe the invariance of $v_2-v_1$ under the symmetric group $S_4$ of degree four
acting on $\mathcal{V}$ as the permutation of the variables.
The symmetric group $S_4$ is generated by transpositions $(1\ 2)$, $(1\ 3)$ and $(1\ 4)$.
Here, we set
\[
\sigma _1:=\begin{pmatrix}
1 & 2 & 3 & 4 \\
2 & 1 & 3 & 4
\end{pmatrix}, \quad \sigma _2:=\begin{pmatrix}
1 & 2 & 3 & 4 \\
3 & 2 & 1 & 4
\end{pmatrix}, \quad \sigma _3:=\begin{pmatrix}
1 & 2 & 3 & 4 \\
4 & 2 & 3 & 1
\end{pmatrix}. 
\]
By the symmetry of the third and fourth variables of $\mu$ and $\nu$, 
we have the following relations:
\[
\sigma _1(v_1) = v_1,\quad \sigma _1(v_2) = v_2,\quad \sigma _2(v_1) = v_1, \quad \sigma _2(v_2) = v_1, \quad \sigma _3(v_1) = v_2 , \quad \sigma _3(v_2) = v_1.
\]
From these, 
it can be seen that $v_1$ and $v_2$ are symmetric by the transposition $(1\ 2)=\sigma_1$, 
and $v_2-v_1$ is antisymmetric by the permutation $\sigma_3$. 
There are totally $24$ elements in $S_4$, however,
due to the invariance by the permutation $\sigma_1$ and the symmetry for the third and fourth variables of $\mu$ and $\nu$, 
the action of $S_4$ on $\mathcal{V}$ is reduced to the following six permutations:
\[
\sigma _1,\; \sigma _2,\; \sigma _3,\;  
\sigma _4:=\begin{pmatrix}
1 & 2 & 3 & 4 \\
3 & 4 & 1 & 2
\end{pmatrix},\; 
\sigma _5:=\begin{pmatrix}
1 & 2 & 3 & 4 \\
1 & 4 & 3 & 2
\end{pmatrix},\;   
\sigma _6:=\begin{pmatrix}
1 & 2 & 3 & 4 \\
1 & 3 & 2 & 4
\end{pmatrix}. 
\]
Then we can verify that $v_2-v_1$ is antisymmetric by the permutations $\sigma_3$ and $\sigma_6$.
Furthermore, 
an element of $\mathcal{V}$ is antisymmetric by $\sigma_3$ and $\sigma_6$
if and only if it is a scalar multiple of $v_2-v_1$.

In a similar way, 
we observe the invariance of the Chern--Federer polynomial under the symmetric group $S_4$. 
First, we rewrite the Chern--Federer polynomial as follows. 
For $H=(h_{ij}^\alpha)\in \mathrm{II}(\mathbb{E}^m_p,\mathbb{E}^n_q)$, 
we define $\rho \in \otimes ^4(\mathbb{E}^m_p)^*$ as follows
\[
\rho := \sum _{i,j,k,l} \rho_{ijkl}\; e^i\otimes e^j\otimes e^k\otimes e^l ,
\]
where $\rho _{ijkl}$ is defined by
\[
\rho _{ijkl} := \sum _\alpha \varepsilon'_\alpha h_{ij}^\alpha h_{kl}^\alpha
\]
and $\{e^i\}_{i=1}^m$ is the dual basis of the standard basis of $\mathbb{E}^m_p$. 
Then we have
\begin{align*}
\mathcal{Q}_1(H) = \sum _\alpha \varepsilon ' _\alpha \sum _{i,j}\varepsilon _i\varepsilon _j h_{ij}^\alpha h_{ij}^\alpha = \sum _{i,,j} \varepsilon _i \varepsilon _j \rho _{ijij} = \sum _{i,j} \rho _{ij}\; ^{ij} = C_{13}C_{24}\rho 
\end{align*}
and
\begin{align*}
\mathcal{Q}_2(H) = \sum _\alpha \varepsilon '_\alpha \sum _{i,j} \varepsilon _i\varepsilon _j h_{ii}^\alpha h_{jj}^\alpha = \sum _{i,j} \varepsilon _i\varepsilon _j\rho _{iijj} = C_{12}C_{34}\rho .
\end{align*}
Therefore we can rewrite the Chern--Federer polynomial $\mathrm{CF}(H)$ as follows: 
\[
\mathrm{CF}(H) = \mathcal{Q}_2(H) - \mathcal{Q}_1(H) = \mathrm{det} \begin{pmatrix}
C_{12} & C_{13} \\
C_{24} & C_{34} 
\end{pmatrix} \rho .
\]
As in the case of $\mathcal{V}$,
the action of $S_4$ on $\mathcal{U}$ is reduced to six elements $\sigma _i\; (i=1,2,\cdots ,6)$. 
Then we can verify that an element of $\mathcal{U}$ is antisymmetric by $\sigma_3$ and $\sigma_6$
if and only if it is a scalar multiple of the Chern--Federer polynomial $\mathrm{CF}(H)$.
Consequently, 
we find that $\mathrm{CF}(H)$ and $v_2-v_1$ have the same symmetry
via the first variational formula and the actions of $S_4$ on $\mathcal{U}$ and $\mathcal{V}$.

\section{Chern--Federer submanifolds in Riemannian space forms}
\label{sec:CF submanifolds}

Let $(M^{m}, g_{M}), \ (N^{n}, g_{N})$ be two Riemannian manifolds. 
From now on, 
we deal with isometric immersions $\varphi : (M^{m}, g_{M}) \rightarrow (N^{n}, g_{N})$. 
In this section, we firstly derive the Euler--Lagrange equation for an isometric immersion from a Riemannian manifold into a Riemannian space form. 
Secondly, we construct examples in the case of curves or surfaces. 
Finally, 
we consider Chern--Federer isoparametric hypersurfaces in Riemannian space forms.  

\subsection{Euler--Lagrange equations for isometric immersions} 
\label{subsec:isometric immersions}

For an isometric immersion $\varphi : (M^{m}, g_{M}) \rightarrow (N^{n}, g_{N})$, 
we denote the shape operator and the mean curvature vector field 
by $A$ and $\mathcal{H}$, respectively. 
Namely, they are defined by 
\begin{equation*}
\langle A_{\xi}(X), Y \rangle 
= \langle (\widetilde{\nabla} d\varphi)(X, Y), \xi \rangle, \quad 
\mathcal{H} = \frac{1}{m} \textrm{tr}_{g_{M}} (\widetilde{\nabla} d\varphi) 
= \frac{1}{m} \tau(\varphi) 
\end{equation*}
for any $X, Y \in \Gamma(TM), \ \xi \in \Gamma(T^{\bot}M)$, 
where $T^{\bot}M$ is the normal bundle over $M$ of $\varphi$. 
In addition, 
we simply denote by $h$ the second fundamental form $\widetilde{\nabla} d\varphi$ in this section. 

We denote a Riemannian space form of constant curvature $c \in \mathbb{R}$ by $N^{n}(c)$. 
Namely, it is locally isometric to one of a Euclidean space ($c = 0$), 
a round sphere ($c > 0$) and a hyperbolic space ($c < 0$). 

When we denote the Ricci operator of $(M^{m}, g_{M})$ by $Q$, 
we obtain the Euler--Lagrange equation for an isometric immersion into a Riemannian space form. 

\begin{thm} \label{CF_isom_imm}
Let $\varphi : (M^{m}, g_{M}) \rightarrow N^{n}(c)$ be an isometric immersion. 
Then $\varphi$ is a Chern--Federer map if and only if it satisfies that 
\begin{equation}
\mathrm{CF}(\varphi) = 
-d\varphi({\mathrm{tr}}_{g_{M}} (\nabla Q)) + 2cm(m-1)\mathcal{H}
 - {\mathrm{tr}}_{g_{M}} h(Q(\textrm{-}), \textrm{-}) = 0, \label{EL_CF_submfd_1}
\end{equation}
equivalently, 
\begin{equation}
(\top) : {\mathrm{tr}}_{g_{M}} (\nabla Q) = 0, \quad 
(\bot) : 2cm(m-1)\mathcal{H} - {\mathrm{tr}}_{g_{M}} h(Q(\textrm{-}), \textrm{-}) = 0, \label{EL_CF_submfd_2}
\end{equation}
where $(\top)$ and $(\bot)$ denote the tangent component and the normal component of (\ref{EL_CF_submfd_1}), respectively. 
\end{thm}

\begin{rem} 
We define two $(1,1)$-type tensor fields $A^{C}$ and $\Xi$ on $M^{m}$ as 
\begin{equation*}
A^{C}(X) := \sum_{\alpha=1}^{k} A_{\xi_{\alpha}}^{2} (X), \quad 
\Xi(X) := A_{\tau(\varphi)}(X) - A^{C}(X) = m A_{\mathcal{H}}(X) - A^{C}(X), 
\end{equation*}
where $k = n-m$ and $\{ \xi_{\alpha} \}_{\alpha = 1}^{k}$ is a local orthonormal frame of $T^{\bot}M$. 
The operator $A^{C}$ is called the \textit{Casorati operator} (cf.~\cite{Chen3, Chen4}). 
Then, from the Gauss equation, we have
\begin{equation*}
Q(X) = c(m-1)X + \Xi(X). 
\end{equation*}
From this, we can also describe the formula (\ref{EL_CF_submfd_2}) as 
\begin{equation}
(\top) : {\mathrm{tr}}_{g_{M}} (\nabla \Xi) = 0, \quad 
(\bot) : cm(m-1)\mathcal{H} - {\mathrm{tr}}_{g_{M}} h(\Xi(\textrm{-}), \textrm{-}) = 0. \label{EL_CF_submfd_3}
\end{equation}
\end{rem} 

\subsubsection*{Proof of Theorem~\ref{CF_isom_imm}} 
Since the target space $N^{n}$ is of constant curvature $c$ and $\varphi^{\ast} g_{N} = g_{M}$, 
by using Lemma~\ref{CFmap_eq}, we compute 
\begin{align*}
&\sum_{i,j=1}^{m} \left\{ 
\left( \nabla R^{N} \right)(d\varphi(e_{i}), d\varphi(e_{i}), d\varphi(e_{j}), d\varphi(e_{j})) 
- d\varphi((\nabla R^{M})(e_{i}, e_{i}, e_{j}, e_{j})) 
\right. \\ 
& 
\left. \quad - h(e_{i}, R^{M}(e_{i}, e_{j})e_{j}) + 2R^{N}(h(e_{i}, e_{i}), d\varphi(e_{j}))d\varphi(e_{j}) 
+ 2R^{N}(d\varphi(e_{i}), h(e_{i}, e_{j}))d\varphi(e_{j}) 
\right\} \\
&= -d\varphi(\textrm{tr}_{g_{M}} (\nabla Q)) 
- \textrm{tr}_{g_{M}} h(Q(\textrm{-}), \textrm{-})
+ 2c(m-1)\tau(\varphi). 
\end{align*}
Therefore, the proof is completed since $\tau(\varphi) = m\mathcal{H}$. \qed

\subsection{Examples of Chern--Federer submanifolds}
\label{subsec:examples}

Here, we construct some examples of Chern--Federer maps in the case of isometric immersions. 
When an isometric immersion $\varphi : (M^{m}, g_{M}) \rightarrow (N^{n}, g_{N})$ is a Chern--Federer map, 
we call the image a \textit{Chern--Federer submanifold} in $(N^{n}, g_{N})$, 
and the map $\varphi$ to be \textit{Chern--Federer}. 

Let $I \subset \mathbb{R}$ be an open interval. 
Then an arbitrary curve $\gamma : I \rightarrow (N^{n}, g_{N})$ 
is a Chern--Federer map. 
Actually, we have $W_{1}(\gamma) = W_{2}(\gamma)$ 
from Theorem~\ref{thmQ_1} and \ref{thmQ_2}. 
Therefore, it is trivial that 
\begin{equation*}
\mathrm{CF}(\gamma) = W_{2}(\gamma) - W_{1}(\gamma) = 0. 
\end{equation*}

There are other obvious examples in the following way.
We consider a Euclidean $n$-space $\mathbb{E}^{n}$ 
as a target space $(N^{n}, g_{M})$, which is a flat Riemannian space form. 
If $(M^{m}, g_{M})$ is a Ricci-flat Riemannian manifold, then an arbitrary isometric immersion $\varphi : (M^{m}, g_{M}) \rightarrow \mathbb{E}^{n}$ is Chern--Federer. 
For example, Calabi--Yau manifolds, Hyperk$\ddot{\textrm{a}}$hler manifolds and $G_{2}$-manifolds are all Ricci-flat. 
Moreover, for any Riemannian manifold $(M^{m}, g_{M})$, there exists an isometric immersion into a Euclidean space by Nash's theorem. 

Next, we consider the two-dimensional case ($m=2$). 
\begin{prop} 
Let $\varphi : (M^{2}, g_{M}) \rightarrow N^{n}(c)$ be an isometric immersion and $K$ the sectional curvature of $(M^{2}, g_{M})$. 
Then $\varphi$ is Chern--Federer if and only if 
\begin{itemize}
\item[(i)] $K$ is constant and $\varphi$ is minimal, or
\item[(ii)] $K = 2c$ and $\varphi$ is arbitrary, that is, unconditional on $\varphi$. 
\end{itemize}
\end{prop}

\begin{proof}
In the two-dimensional case, we have, for any $X \in \Gamma(TM)$, 
\begin{equation*}
Q(X) = K X.
\end{equation*}
Thus, since $\varphi$ is Chern--Federer if and only if 
\begin{align*}
(\top) \ &: \ \textrm{tr}_{g_{M}} (\nabla Q) = \textrm{grad} \, K = 0, \\
(\bot) \ &: \ 4c\mathcal{H} - K \textrm{tr}_{g_{M}} h(\textrm{-}, \textrm{-}) 
= 2(2c - K) \mathcal{H} = 0,
\end{align*}
we have the conclusion. 
\end{proof}

Let $M^{2}(K)$ be a two-dimensional Riemannian space form of constant curvature $K$. 
For minimal isometric immersions $\varphi : M^{2}(K) \rightarrow N^{n}(c)$, 
the research has already completed. 
In fact, 
\begin{itemize}
\item when $c = 0$, it implies that $K = 0$ and $\varphi$ is totally geodesic; 
\item when $c = -1$, it implies that $K = -1$ and $\varphi$ is totally geodesic; 
\item when $c = 1$, it implies that $K \geq 0$. 
In addition, if $N^{n}(1)$ is isometric to a round sphere 
\begin{equation*}
\mathbb{S}^{n}(1) := \{(x_{1}, \cdots, x_{n+1}) \in \mathbb{E}^{n+1} \mid x_{1}^{2} + \cdots + x_{n+1}^{2} = 1 \},
\end{equation*}
then $\varphi$ is locally congruent to generalized Clifford tori, or 
Bor$\mathring{\textrm{u}}$vka spheres $\psi_{k} \ (k \geq 1)$, which are nothing but the standard minimal immersions of two-dimensional spheres. 
See \cite{Bry85, Kem76} in detail. 
\end{itemize}

At the end of Section~\ref{subsec:examples}, 
we consider flat tori in the unit $3$-sphere $\mathbb{S}^{3}(1)$. 

Let $T^{2}$ be a flat torus, $\varphi : T^{2} \rightarrow \mathbb{S}^{3}(1)$ an isometric immersion. 
Then the flat torus $T^{2}$ admits an asymptotic Chebyshev net $(s_{1}, s_{2})$, 
that is, by using the asymptotic Chebyshev net $(s_{1}, s_{2})$, we can express 
\begin{equation*}
g_{T} = ds_{1}^{2} + 2\cos{\omega} \, ds_{1}ds_{2} + ds_{2}^{2}, 
\quad 
h_{T} = 2\sin{\omega} \, ds_{1}ds_{2}, 
\end{equation*}
where $\omega = \omega(s_{1}, s_{2})$ is some smooth function, 
and $g_{T}, h_{T}$ are the induced metric and the second fundamental form 
of $\varphi$, respectively. 
Moreover, we compute the mean curvature function $\mathcal{H}$ of $\varphi$ from this as 
\begin{equation*}
\mathcal{H}(s_{1}, s_{2}) = -\cot{[\omega(s_{1}, s_{2})]}. 
\end{equation*}
See \cite{Ki1} in more precise details regarding an asymptotic Chebyshev net of a flat torus. 

\begin{thm} \label{flat_tori}
Let $T^{2}$ be a flat torus, $\varphi : T^{2} \rightarrow \mathbb{S}^{3}(1)$ an isometric immersion with constant mean curvature $\mathcal{H}$. 
Then $\varphi$ is an $(\alpha \mathcal{Q}_{1} + \beta \mathcal{Q}_{2})$-map 
if and only if 
\begin{itemize}
\item[(\textrm{i})]  $\mathcal{H} = 0$ \quad $(\textrm{when} \ \alpha + \beta = 0)$, 
\item[(\textrm{ii})] $\mathcal{H} = 0$ \quad 
$\left(\textrm{when} \ \alpha + \beta \neq 0, \ 
\dfrac{\alpha}{\alpha + \beta} \geq 0 \right)$, 
\item[(\textrm{iii})] $\mathcal{H} = 0$, or 
$\mathcal{H}^{2} = -\dfrac{\alpha}{2(\alpha + \beta)}$ \quad 
$\left(\textrm{when} \ \alpha + \beta \neq 0, \ 
\dfrac{\alpha}{\alpha + \beta} < 0 \right)$. 
\end{itemize}
Moreover, in the case of (iii), 
$\mathcal{H}^{2}$ runs across the whole range of $(0, \infty)$. 
\end{thm}

In \cite{Ki2}, Kitagawa showed that any isometric embedding $\varphi : T^{2} \rightarrow \mathbb{S}^{3}(1)$ with constant mean curvature are congruent to Clifford tori. 
Therefore, we have the following classification theorem. 

\begin{cor} 
Let $T^{2}$ be a flat torus, $\varphi : T^{2} \rightarrow \mathbb{S}^{3}(1)$ an isometric embedding with constant mean curvature $\mathcal{H}$. 
Then $\varphi$ is an $(\alpha \mathcal{Q}_{1} + \beta \mathcal{Q}_{2})$-map 
if and only if it is congruent to one of the following Clifford tori 
\begin{itemize}
\item[(i)] a minimal Clifford torus defined by 
\begin{equation*}
\mathbb{S}^{1}\left(\frac{1}{\sqrt{2}} \right) \times \mathbb{S}^{1}\left(\frac{1}{\sqrt{2}} \right) \hookrightarrow \mathbb{S}^{3}(1) \quad (\textrm{when} \ \alpha + \beta = 0), 
\end{equation*}
\item[(ii)] a minimal Clifford torus defined by 
\begin{equation*}
\mathbb{S}^{1}\left(\frac{1}{\sqrt{2}} \right) \times \mathbb{S}^{1}\left(\frac{1}{\sqrt{2}} \right) \hookrightarrow \mathbb{S}^{3}(1) \quad \left(\textrm{when} \ 
\alpha + \beta \neq 0, \ \frac{\alpha}{\alpha + \beta} \geq 0 \right),
\end{equation*}
\item[(iii)] a minimal Clifford torus defined by 
\begin{equation*}
\mathbb{S}^{1}\left(\frac{1}{\sqrt{2}} \right) \times \mathbb{S}^{1}\left(\frac{1}{\sqrt{2}} \right) \hookrightarrow \mathbb{S}^{3}(1), 
\end{equation*}
or a non-minimal Clifford torus defined by 
\begin{equation*}
\mathbb{S}^{1}(r_{1}) \times \mathbb{S}^{1}(r_{2}) \hookrightarrow \mathbb{S}^{3}(1) \quad \left(\textrm{when} \ \alpha + \beta \neq 0, \ \frac{\alpha}{\alpha + \beta} < 0 \right), 
\end{equation*}
where $r_{1}, r_{2}$ are defined by 
\begin{align*}
r_{1} &= \frac{1}{2}\left[\sqrt{1 + \sqrt{\frac{2(\alpha + \beta)}{\alpha + 2\beta}}} - \sqrt{1 - \sqrt{\frac{2(\alpha + \beta)}{\alpha + 2\beta}}} \ \right], \\ 
r_{2} &= \frac{1}{2}\left[\sqrt{1 + \sqrt{\frac{2(\alpha + \beta)}{\alpha + 2\beta}}} + \sqrt{1 - \sqrt{\frac{2(\alpha + \beta)}{\alpha + 2\beta}}} \ \right], 
\end{align*}
and the mean curvature of the Clifford torus $\mathbb{S}^{1}(r_{1}) \times \mathbb{S}^{1}(r_{2}) \hookrightarrow \mathbb{S}^{3}(1)$ satisfies that 
\begin{equation*}
\mathcal{H}^{2} = -\dfrac{\alpha}{2(\alpha + \beta)}. 
\end{equation*}
\end{itemize}
\end{cor} 

\subsubsection*{Proof of Theorem~\ref{flat_tori}}
Let $(s_{1}, s_{2})$ be an asymptotic Chebyshev net for $T^{2}$. 
We define a frame field by using this coordinates 
\begin{equation*}
e_{1} = \frac{\partial}{\partial s_{1}}, \quad 
e_{2} = \mathcal{H}\dfrac{\partial}{\partial s_{1}} 
+ \sqrt{1 + \mathcal{H}^{2}}\dfrac{\partial}{\partial s_{2}}.
\end{equation*}
Then $\{e_{1}, e_{2}\}$ defines a geodesic frame. 
By using this, we compute 
\begin{equation*}
W_{1}(\varphi) 
= -4\mathcal{H}(1 + 2\mathcal{H}^{2}) \xi, \quad W_{2}(\varphi) 
= -8\mathcal{H}^{3} \xi, 
\end{equation*}
where $\xi$ is a unit normal vector along $\varphi$. 
Namely, we have 
\begin{equation*}
\alpha W_{1}(\varphi) + \beta W_{2}(\varphi) 
= -4\mathcal{H}\{\alpha + 2(\alpha + \beta)\mathcal{H}^{2} \} \xi. 
\end{equation*}
This completes the proof. 
\qed

\begin{rem} 
Regarding the following hypersurfaces in unit spheres 
\begin{itemize}
\item $\mathbb{S}^{m}\left(\dfrac{1}{\sqrt{2}}\right) \subset \mathbb{S}^{m+1}(1)$ 
(a totally umbilical small sphere), 
\item $\mathbb{S}^{m}\left(\dfrac{1}{\sqrt{2}}\right) \times \mathbb{S}^{m}\left(\dfrac{1}{\sqrt{2}}\right) \subset \mathbb{S}^{2m+1}(1)$ 
(a minimal generalized Clifford torus), 
\end{itemize}
these inclusion maps are both $(\alpha \mathcal{Q}_{1} + \beta \mathcal{Q}_{2})$-maps for any $\alpha, \beta \in \mathbb{R}$ such that 
$\alpha^{2} + \beta^{2} \neq 0$. 
\end{rem} 

\subsection{Chern--Federer isoparametric hypersurfaces in space forms}
\label{subsec:CF isoparametric hypersurfaces}

We remark that for a hypersurface $M^{m} \subset N^{m+1}$ 
with a unit normal vector field $\xi$, it holds that 
\begin{equation}
h(X, Y) = \langle A_{\xi}(X), Y \rangle \xi \label{hsurf_2nd_ff}
\end{equation}
for any $X, Y \in \Gamma(TM)$, 
and we may denote the shape operator $A_{\xi}$ by $A$ for simplicity. 

Let $M^{m} \subset N^{m+1}(c)$ be an isoparametric hypersurface, that is, a hypersurface with constant principal curvatures. 
Then the inclusion map $\iota : M^{m} \hookrightarrow N^{m+1}(c)$ gives an isometric immersion by considering the induced metric $g_{M}$ by $\iota$, 
and we have an orthogonal direct sum decomposition as vector bundles 
\begin{equation*}
TM = \bigoplus_{\alpha=1}^{g} E_{\alpha}, 
\end{equation*}
where $g$ denotes the number of distinct principal curvatures and $E_{\alpha}$ are the principal (curvature) distributions. 
We remark that each $E_{\alpha}$ is auto-parallel, that is, 
the following holds 
\begin{equation*}
\nabla_{X} Y \in \Gamma(E_{\alpha}) \quad (X, Y \in \Gamma(E_{\alpha})), 
\end{equation*}
where $\nabla$ denotes the Levi--Civita connection of $(M^{m}, g_{M})$. 
In particular, each $E_{\alpha}$ is integrable. 
More precisely, see \cite[Lemma~3.9]{CR} in detail. 

\begin{thm} \label{isopara} 
Let $M^{m} \subset N^{m+1}(c)$ be an isoparametric hypersurface in a Riemannian space form. 
Then $M^{m}$ is Chern--Federer if and only if it satisfies that 
\begin{equation*}
c(m-1) ({\mathrm{tr}} \, A) - ({\mathrm{tr}} \, A)({\mathrm{tr}} \, A^{2}) 
+ ({\mathrm{tr}} \, A^{3}) = 0.
\end{equation*}
\end{thm}

\begin{rem}
Let $M^{m} \subset N^{m+1}(c)$ be an isoparametric hypersurface. 
Then the inclusion map $\iota$ is a $\mathcal{Q}_{1}$-map if and only if 
\begin{equation*}
W_{1}(\iota) = c(\textrm{tr} \, A) - (\textrm{tr} \, A^{3}) = 0, 
\end{equation*}
the inclusion map is a $\mathcal{Q}_{2}$-map (that is, a biharmonic map) if and only if 
\begin{equation*}
W_{2}(\iota) =  \left(mc - (\textrm{tr} \, A^{2})\right) (\textrm{tr} \, A) = 0, 
\end{equation*}
and the inclusion map is a Willmore--Chen map if and only if 
\begin{equation*}
\mathrm{WC}(\iota) = mW_{1}(\iota) - W_{2}(\iota) = (\textrm{tr} \, A)(\textrm{tr} \, A^{2}) 
-m(\textrm{tr} \, A^{3}) = 0. 
\end{equation*}
\end{rem}

\begin{lem} \label{isopara_lem1} 
Let $M^{m} \subset N^{m+1}(c)$ be an isoparametric hypersurface, 
$\iota : M^{m} \hookrightarrow N^{m+1}(c)$ the inclusion map 
and $g_{M}$ the induced metric of $M^{m}$ by $\iota$. 
Then it holds that  
\begin{equation*}
{\mathrm{tr}}_{g_{M}} (\nabla \Xi) = 0. 
\end{equation*}
\end{lem} 

\begin{proof}
Let $\{e_{i}\}_{i=1}^{m}$ be an orthonormal frame of $M^{m}$ such that 
\begin{equation*}
A(e_{i}) = \lambda_{i} e_{i}, 
\end{equation*}
where $\lambda_{i}$'s are principal curvatures, which are constant. 
Then we have by using (\ref{hsurf_2nd_ff})
\begin{align*}
\textrm{tr}_{g_{M}} (\nabla \Xi) &= 
\sum_{k=1}^{m} \langle \textrm{tr}_{g_{M}} (\nabla \Xi), e_{k} \rangle e_{k} \\
&= \sum_{i,j,k=1}^{m} \left[ 
\left \langle \nabla_{e_{i}} (A_{h(e_{j}, e_{j})} e_{i}) - (A_{h(e_{j}, e_{j})} \nabla_{e_{i}} e_{i}), e_{k} \right \rangle \right. \\ 
&\qquad \qquad \qquad  \quad 
\left. - \left \langle \nabla_{e_{i}} (A_{h(e_{i}, e_{j})} e_{j}) - (A_{h(\nabla_{e_{i}} e_{i}, e_{j})} e_{j}), e_{k} \right \rangle \right] e_{k} \\
&= \sum_{i,j,k=1}^{m} \left[
-\lambda_{i}\lambda_{j} \delta_{ij} \langle \nabla_{e_{i}}e_{j}, e_{k} \rangle 
+ \lambda_{i}\lambda_{j} \delta_{jk} \langle \nabla_{e_{i}}e_{i}, e_{j} \rangle
 \right] e_{k} \\
&= \sum_{i,j=1}^{m} (\lambda_{i}\lambda_{j} - \lambda_{i}^{2})\langle \nabla_{e_{i}}e_{i}, e_{j} \rangle e_{j}.
\end{align*} 
From the last formula, 
we can claim the following statements 
for $e_{i} \in \Gamma(E_{\alpha}), \, e_{j} \in \Gamma(E_{\beta})$: 
When $\alpha = \beta$, we have 
\begin{equation*}
\lambda_{i} \lambda_{j} - \lambda_{i}^{2} = 0
\end{equation*}
since $\lambda_{i} = \lambda_{j}$. 
When $\alpha \neq \beta$, we have 
\begin{equation*}
\langle \nabla_{e_{i}} e_{i}, e_{j} \rangle = 0
\end{equation*}
since $\nabla_{e_{i}} e_{i} \in \Gamma(E_{\alpha})$ 
and $E_{\alpha}$ is orthogonal to $E_{\beta}$. 
Therefore, we complete the proof. 
\end{proof}

\begin{lem} \label{isopara_lem2} 
Under the assumption of Lemma~\ref{isopara_lem1}, it holds that 
\begin{equation*}
{\mathrm{tr}}_{g_{M}} h(\Xi(\textrm{-}), \textrm{-}) = 
[({\mathrm{tr}} \, A) ({\mathrm{tr}} \, A^{2}) - ({\mathrm{tr}} \, A^{3})]\xi,
\end{equation*}
where $\xi$ is a unit normal vector field of $M^{m}$. 
\end{lem} 

\begin{proof}
Taking an orthonormal frame $\{ e_{i} \}_{i=1}^{m}$ of $M^{m}$ such that 
$A(e_{i}) = \lambda_{i} e_{i}$, we compute by using (\ref{hsurf_2nd_ff}) that 
\begin{align*}
\textrm{tr}_{g_{M}} h(\Xi(\textrm{-}), \textrm{-}) &= 
\sum_{i,j=1}^{m} h(A_{h(e_{j}, e_{j})} e_{i} - A_{h(e_{i}, e_{j})} e_{j}, e_{i}) \\
&= \sum_{i,j=1}^{m} h(e_{i}, \langle A(e_{j}), e_{j} \rangle A(e_{i}) - \langle A(e_{i}), e_{j} \rangle A(e_{j}) ) \\
&= \sum_{i,j=1}^{m} \left[\lambda_{i}^{2} \lambda_{j} - \lambda_{i}^{2} \lambda_{j} \delta_{ij}^{2} \right] \xi 
= [(\textrm{tr} \, A) (\textrm{tr} \, A^{2}) - (\textrm{tr} \, A^{3})]\xi. 
\end{align*}
Thus, the proof is completed. 
\end{proof}

\subsubsection*{Proof of Theorem~\ref{isopara}}
From Lemma~\ref{isopara_lem1} and Lemma~\ref{isopara_lem2}, 
we can see that 
an isoparametric hypersurface $M^{m} \subset N^{m+1}(c)$ is Chern--Federer if and only if it holds that 
\begin{align*}
(\top) \ &: \ \textrm{tr}_{g_{M}} (\nabla \Xi) = 0 \quad (\textrm{trivially holds}), \\
(\bot) \ &: \ cm(m-1)\mathcal{H} - \textrm{tr}_{g_{M}} h(\Xi(\textrm{-}), \textrm{-}) 
= \left[c(m-1) (\textrm{tr} \, A) - (\textrm{tr} \, A) (\textrm{tr} \, A^{2}) 
+ (\textrm{tr} \, A^{3}) \right] \xi = 0. 
\end{align*}
Thus, we obtain the conclusion. 
\qed

Let $\mathbb{L}^{n}$ be a Minkowski $n$-space. 
By using the classification \cite[Theorem~3.12, Theorem~3.14]{CR} of isoparametric hypersurfaces in a Euclidean space $\mathbb{E}^{m+1}$ and a hyperbolic space 
\begin{equation*}
\mathbb{H}^{m+1}(-1) := \{ (x_{1}, \cdots, x_{m+2}) \in \mathbb{L}^{m+2} \mid -x_{1}^{2} + x_{2}^{2} + \cdots + x_{m+2}^{2} = -1, \ x_{1} > 0 \}, 
\end{equation*}  
we have the following results: 

\begin{thm} 
Let $M^{m} \subset \mathbb{E}^{m+1}$ be an isoparametric hypersurface. 
Then $M^{m}$ is Chern--Federer if and only if 
it is congruent to an open portion of one of the following hypersurfaces 
\begin{itemize}
\setlength{\leftskip}{0.5cm}
\item[${\mathrm{[}}g = 1{\mathrm{]}}$] \quad 
$\mathbb{E}^{m} \subset \mathbb{E}^{m+1}$ \ 
$(\textrm{a totally geodesic hyperplane} \, )$, 
\item[${\mathrm{[}}g = 2{\mathrm{]}}$] \quad 
$\mathbb{S}^{1}(r) \times \mathbb{E}^{m-1} \subset \mathbb{E}^{m+1}$ \ 
$(\textrm{a generalized right circular cylinder} \, )$. 
\end{itemize}
\end{thm} 

\begin{thm} 
Let $M^{m} \subset \mathbb{H}^{m+1}(-1)$ be an isoparametric hypersurface. 
Then $M^{m}$ is Chern--Federer if and only if it is totally geodesic. 
\end{thm}

In the case of a unit sphere $\mathbb{S}^{m+1}(1)$, 
there exist fruitfully Chern--Federer isoparametric hypersurfaces 
which is not minimal. 
This is a different situation from that of biharmonic isoparametric hypersurfaces in a unit sphere. 
See \cite{IIU} on the classification of biharmonic isoparametric hypersurfaces. 
In this paper, 
we do not classify Chern--Federer isoparametric hypersurfaces 
in $\mathbb{S}^{m+1}(1)$. 
However, we show some examples of Chern--Federer homogeneous hypersurfaces, which are also isoparametric. 
Since all of their proofs are done by direct calculations by using Theorem~\ref{isopara}, detailed calculations are omitted. 
We again remark that $g$ denotes the number of distinct principal curvatures of isoparametric hypersurfaces. 

\subsubsection*{$\bullet$ ${\mathrm{[}}g = 1{\mathrm{]}}$}
The classification is the following totally umbilical hypersurfaces
\begin{equation}
\mathbb{S}^{m}(r) = \left\{ (x, \sqrt{1-r^{2}}) \in \mathbb{E}^{m+2} \mid 
||x||^{2} = r^{2} \right\} \subset \mathbb{S}^{m+1}(1)
\quad (0 < r \leq 1), 
\label{g1}
\end{equation}
where $||\cdot||$ denotes the canonical Euclidean norm of $\mathbb{E}^{m+1}$. 
From this, we obtain: 

\begin{prop} 
The isoparametric hypersurface (\ref{g1}) is Chern--Federer 
if and only if $r = 1$ (totally geodesic one), 
or $r = 1/\sqrt{2}$ (proper biharmonic one). 
\end{prop}

\subsubsection*{$\bullet$ ${\mathrm{[}}g = 2{\mathrm{]}}$}
The classification is the following Clifford hypersurfaces 
\begin{equation}
\mathbb{S}^{p}(r_{1}) \times \mathbb{S}^{m-p}(r_{2}) \subset \mathbb{S}^{m+1}(1) \quad (r_{1}^{2} + r_{2}^{2} = 1). 
\label{g2}
\end{equation}
We denote the distinct principal curvatures of (\ref{g2}) by $\lambda_{1}, \lambda_{2}$. 
Then by setting 
\begin{equation*}
\lambda := \lambda_{1} = \cot{t} \quad \left(0 < t < \frac{\pi}{2} \right),
\end{equation*}
we have 
\begin{equation*}
\lambda_{2} = \cot{\left(t + \frac{\pi}{2} \right)} 
= -\frac{1}{\cot{t}} = -\frac{1}{\lambda}. 
\end{equation*}
From this, we obtain: 

\begin{prop} 
The isoparametric hypersurface (\ref{g2}) is Chern--Federer 
if and only if $\lambda$ satisfies that 
\begin{equation*}
p(p-1)\lambda^{6} 
- p(2m - p -1)\lambda^{4} 
+ (m-p)(m+p-1)\lambda^{2} 
- (m-p)(m-p-1) = 0.
\end{equation*}
\end{prop}

\subsubsection*{$\bullet$ ${\mathrm{[}}g = 3{\mathrm{]}}$} 
The classification is the following four Cartan hypersurfaces 
\begin{align}
M^{3} &= {SO(3)}/{\mathbb{Z}_{2} \times \mathbb{Z}_{2}} \rightarrow \mathbb{S}^{4}(1), \label{g3_1} \\
M^{6} &= {SU(3)}/{T^{2}} \rightarrow \mathbb{S}^{7}(1), \label{g3_2} \\
M^{12} &= {Sp(3)}/{Sp(1)^{3}} \rightarrow \mathbb{S}^{13}(1), \label{g3_3} \\
M^{24} &= {F_{4}}/{{Spin}(8)} \rightarrow \mathbb{S}^{25}(1). \label{g3_4} 
\end{align}
We denote the distinct principal curvatures of (\ref{g3_1}--\ref{g3_4}) 
by $\lambda_{1}, \lambda_{2}, \lambda_{3}$. 
Then by setting 
\begin{equation*}
\lambda := \lambda_{1} = \cot{t} \quad \left(0 < t < \frac{\pi}{3} \right),
\end{equation*}
we have 
\begin{equation*}
\lambda_{2} = \frac{\lambda - \sqrt{3}}{\sqrt{3} \lambda + 1}, \quad 
\lambda_{3} = -\frac{\lambda + \sqrt{3}}{\sqrt{3} \lambda - 1}.  
\end{equation*}
From this, we obtain: 

\begin{prop} 
The isoparametric hypersurfaces (\ref{g3_1}), (\ref{g3_3}) or (\ref{g3_4}) are Chern--Federer if and only if $\lambda = \sqrt{3}$ (the only minimal one). 

The isoparametric hypersurface (\ref{g3_2}) is Chern--Federer 
if and only if  $\lambda$ satisfies that 
\begin{equation*}
(\lambda^{2} - 3)(3\lambda^{3} - 3\lambda^{2} - 9\lambda + 1)
(3\lambda^{3} + 3\lambda^{2} - 9\lambda - 1) = 0.
\end{equation*}
Namely, there are non-minimal ones in the case. 
\end{prop}

\subsubsection*{$\bullet$ ${\mathrm{[}}g = 4{\mathrm{]}}$}
In this case, we deal with homogeneous hypersurfaces. 
Non-homogeneous isoparametric ones are called to be of \textit{OT--FKM type}. 
The classification of homogeneous hypersurfaces is the following ones 
\begin{align}
M^{8} &= {SO(5)}/{T^{2}} \rightarrow \mathbb{S}^{9}(1), \label{g4_1} \\
M^{18} &= {U(5)}/{SU(2) \times SU(2) \times U(1)} \rightarrow \mathbb{S}^{19}(1), \label{g4_2} \\
M^{30} &= {U(1) \cdot {{Spin}(10)}}/{S^{1} \cdot {{Spin}(6)}} \rightarrow 
\mathbb{S}^{31}(1), \label{g4_3} \\
M^{4m-2} &= {S(U(2) \times U(m))}/{S(U(1) \times U(1) \times U(m-2))} \rightarrow \mathbb{S}^{4m-1}(1) \quad (m \geq 2), \label{g4_4} \\ 
M^{2m-2} &= {SO(2) \times SO(m)}/{\mathbb{Z}_{2} \times SO(m-2)} \rightarrow \mathbb{S}^{2m-1}(1) \quad (m \geq 3), \label{g4_5} \\
M^{8m-2} &= {Sp(2) \times Sp(m)}/{Sp(1) \times Sp(1) \times Sp(m-2)} \rightarrow \mathbb{S}^{8m-1}(1) \quad (m \geq 2). \label{g4_6}
\end{align}
We denote the distinct principal curvatures of (\ref{g4_1}--\ref{g4_6}) 
by $\lambda_{1}, \lambda_{2}, \lambda_{3}, \lambda_{4}$. 
Then by setting 
\begin{equation*}
\lambda := \lambda_{1} = \cot{t} \quad \left(0 < t < \frac{\pi}{4} \right),
\end{equation*}
we have 
\begin{equation*}
\lambda_{2} = \frac{\lambda - 1}{\lambda + 1}, \quad 
\lambda_{3} = -\frac{1}{\lambda}, \quad 
\lambda_{4} = -\frac{\lambda + 1}{\lambda - 1}. 
\end{equation*}
From this, we obtain: 

\begin{prop} 
The isoparametric hypersurface (\ref{g4_1}) is Chern--Federer 
if and only if  $\lambda = 1+\sqrt{2}$ (the only minimal one). 

The isoparametric hypersurface (\ref{g4_2}) is Chern--Federer 
if and only if $\lambda$ satisfies that 
\begin{equation*}
3\lambda^{12} - 40\lambda^{10} + 223\lambda^{8} - 692\lambda^{6} 
+ 223\lambda^{4} - 40\lambda^{2} + 3 = 0, 
\end{equation*}
which is not minimal. 

The isoparametric hypersurface (\ref{g4_3}) is Chern--Federer 
if and only if $\lambda$ satisfies that 
\begin{equation*}
12\lambda^{12} - 111\lambda^{10} + 488\lambda^{8} - 1098\lambda^{6} 
+ 488\lambda^{4} - 111\lambda^{2} + 12 = 0, 
\end{equation*}
which is not minimal. 

The isoparametric hypersurface (\ref{g4_4}) is Chern--Federer 
if and only if $\lambda$ satisfies that 
\begin{align*}
\lambda^{12} - 4(2m-1)\lambda^{10} + (72m-&85)\lambda^{8} 
- 32(4m^{2}-10m+7)\lambda^{6} \\
\qquad \qquad \qquad \qquad 
&+ (72m-85)\lambda^{4} - 4(2m-1)\lambda^{2} + 1 = 0.
\end{align*}

The isoparametric hypersurface (\ref{g4_5}) is Chern--Federer 
if and only if $\lambda$ satisfies that 
\begin{equation*}
(2m-3)\lambda^{8} - 4(5m-9)\lambda^{6} 
+ 2(16m^{2} - 62m + 63)\lambda^{4} 
- 4(5m-9)\lambda^{2} + 2m - 3 = 0. 
\end{equation*}

The isoparametric hypersurface (\ref{g4_6}) is Chern--Federer 
if and only if $\lambda$ satisfies that 
\begin{align*}
3\lambda^{12} - 16m\lambda^{10} + (136m-&117)\lambda^{8} 
- 4(64m^{2}-116m+63)\lambda^{6} \\
\qquad \qquad \qquad \qquad 
&+ (136m-117)\lambda^{4} - 16m\lambda^{2} + 3 = 0.
\end{align*}
\end{prop}

\subsubsection*{$\bullet$ ${\mathrm{[}}g = 6{\mathrm{]}}$} 
The classification is the following two homogeneous hypersurfaces 
\begin{align}
M^{6} &= {SO(4)}/{\mathbb{Z}_{2} \times \mathbb{Z}_{2}} \rightarrow \mathbb{S}^{7}(1), \label{g6_1} \\
M^{12} &= {G_{2}}/{T^{2}} \rightarrow \mathbb{S}^{13}(1). \label{g6_2} 
\end{align}
We denote the distinct principal curvatures of (\ref{g6_1}), (\ref{g6_2}) 
by $\lambda_{1}, \lambda_{2}, \lambda_{3}, \lambda_{4}, \lambda_{5}, \lambda_{6}$. 
Then by setting 
\begin{equation*}
\lambda := \lambda_{1} = \cot{t} \quad \left(0 < t < \frac{\pi}{6} \right),
\end{equation*}
we have 
\begin{equation*}
\lambda_{2} = \frac{\sqrt{3} \lambda - 1}{\lambda + \sqrt{3}}, \ 
\lambda_{3} = \frac{\lambda - \sqrt{3}}{\sqrt{3} \lambda + 1}, \ 
\lambda_{4} = -\frac{1}{\lambda}, \ 
\lambda_{5} = -\frac{\lambda + \sqrt{3}}{\sqrt{3}\lambda - 1}, \ 
\lambda_{6} = -\frac{\sqrt{3}\lambda + 1}{\lambda - \sqrt{3}}. 
\end{equation*}
From this, we obtain: 

\begin{prop} 
The isoparametric hypersurfaces (\ref{g6_1}) or (\ref{g6_2}) are Chern--Federer 
if and only if $\lambda = 2+\sqrt{3}$ (the only minimal one). 
\end{prop}


\end{document}